\documentclass[12pt]{article}
\usepackage{indentfirst}
\usepackage{caption}
\usepackage{subcaption}
\usepackage{setspace}
\usepackage{graphicx}

\usepackage{algorithm}
\usepackage{amssymb}
\usepackage{amsmath} 
\usepackage{amsthm}

\usepackage{mathrsfs}
\usepackage{mathtools}
\usepackage{listings}
\usepackage{multirow}
\usepackage{comment}

\usepackage{tikz}
\usepackage{tikz-cd}
\usepackage[all,pdf]{xy}

\usepackage[runin]{abstract}
\usepackage{algorithmic}
\usepackage{array}
\usepackage{lipsum}
\usepackage{hyperref}

\usepackage{chngcntr}
\counterwithin{algorithm}{section}

\usepackage{color}

\usepackage{listings}
\renewcommand{\vec}[1]{\boldsymbol{#1}}

\newtheorem{thm}{Theorem}[section]

\newtheorem{defn}{Definition}[section]
\newtheorem{lem}[thm]{Lemma}
\newtheorem{rmk}[thm]{Remark}
\newtheorem{prop}[thm]{Proposition}
\newtheorem{cor}[thm]{Corollary}

\numberwithin{equation}{section}

\newcommand{\RR}{\mathbb{R}}  
\newcommand{\ZZ}{\mathbb{Z}}

\newcommand{\CC}{\mathbb{C}}

\newcommand{\p}{\partial}

\newcommand{\Acal}{\mathcal{A}}
\newcommand{\Bcal}{\mathcal{B}}
\newcommand{\Ccal}{\mathcal{C}}
\newcommand{\Dcal}{\mathcal{D}}
\newcommand{\Ecal}{\mathcal{E}}
\newcommand{\Fcal}{\mathcal{F}}

\newcommand{\Hcal}{\mathcal{H}}

\newcommand{\Lcal}{\mathcal{L}}

\newcommand{\Ncal}{\mathcal{N}}

\newcommand{\Pcal}{\mathcal{P}}

\newcommand{\image}{\text{Im}}

 \newcommand{\wnk}{\vec{\mathrm{k}}}

\newcommand{\myvec}[1]%
{\stackrel{\raisebox{-2pt}[0pt][0pt]{\small$\rightharpoonup$}}{#1}}

\providecommand{\keywords}[1]
{
	\small	
	\textbf{Keywords} #1
}

\newcommand{\grad}{\mathcal{GRAD}}
\newcommand{\curl}{\mathcal{CURL}}
\newcommand{\dive}{\mathcal{DIV}}

\allowdisplaybreaks

\begin{document}
%

\title{Kernel compensation method for Maxwell eigenproblem in photonic crystals with mimetic finite difference discretizations} 
\author{Chenhao Jin\thanks{School of Mathematical Sciences, University of Science and Technology of China, Hefei, Anhui 230026, P.R. China.  E-mail: kanit@mail.ustc.edu.cn.}
	\and Yinhua Xia\thanks{School of Mathematical Sciences, University of Science and Technology of China, Hefei, Anhui 230026, P.R. China.  E-mail: yhxia@ustc.edu.cn. The research of Yinhua Xia was partially supported by National Key R\&D Program of China No. 2022YFA1005202/2022YFA1005200 and NSFC grant No. 12271498. 
	}
	\and Yan Xu\thanks{Corresponding author. School of Mathematical Sciences, University of Science and Technology of China, Hefei, Anhui 230026, P.R. China and Laoshan Laboratory, Qingdao 266237, P.R. China.  E-mail: yxu@ustc.edu.cn. The research of Yan Xu was partially supported by NSFC grant No. 12071455  and Laoshan Laboratory  (No.LSKJ202300305).
	}
}

\date{} 

\maketitle 

\begin{abstract}
We present a kernel compensation method for Maxwell eigenproblem in photonic crystals to avoid the infinite-dimensional kernels that cause many difficulties in the calculation of energy gaps.  The quasi-periodic problem is first transformed into a periodic one on the cube by the Floquet-Bloch theory. Then the compensation operator is introduced in Maxwell's equation with the shifted curl operator. The discrete problem depends on the compatible discretization of the de Rham complex, which is implemented by the mimetic finite difference method in this paper. We prove that the compensation term of the discretization exactly fills up the kernel of the original diecrete problem and avoids spurious eigenvalues. Also, we propose an efficient preconditioner and its FFT and multigrid solvers, which allow parallel computing. Numerical experiments for different three-dimensional lattices in photonic crystals are performed to validate the accuracy and effectiveness of the method. 
\end{abstract}



\section{Introduction}
Photonic crystals (PCs) are lattice-like periodic dielectric structures. With a specific anisotropic material, a PC can have cutoff frequencies that prohibit light transmission in all directions. This property of PCs has wide applications in lasers, filters, and optical transistors \cite{PCs_mold}. The effective calculation of the band gap, as well as the analysis of the energy gap, has always played an essential role both in the theory of physical optics and in practical manufacturing.
The behavior of PCs can be described by a time-harmonic Maxwell's equation consisting of a double curl operator. Based on the Floquet-Bloch theory, the eigenfunctions should satisfy the quasi-periodic boundary condition.

Previous numerical methods can be divided into two types according to the treatment of such boundary conditions. The first type involves the direct discretization of the double curl operator using finite difference methods, such as Yee's scheme \cite{yee1966} or the mimetic finite difference (MFD) method \cite{MFDM2016}.  However, the treatment of the preconditioning matrix requires a delicate linear solver. The second type involves the transformation of the quasi-periodic condition into a periodic condition, which is then compatible with finite element methods \cite{hesthaven2004high, Huoyuan2022, lu2021auxiliary, multilevel_mixed_eigenvalue} and plane wave expansion methods \cite{boffi2006modified, hsue2005extended, white2002computing, johnson2001block}  based on variational formulations.  This transformation, employed in the computation of PCs and quantum mechanics, e.g. \cite{boffi2006modified, lu2021auxiliary}, introduces the concept of the shifted Nabla operator.

The null space of the curl operator in Maxwell's equations, which contains all gradient fields, has an infinite dimension. This poses another challenge in numerical discretization, as the null space of the discrete eigenproblem grows with grid refinement. However, in practice, only the low-energy bands or the first few non-zero frequencies are of interest. Several approaches have been proposed to address the issue of the null space. In \cite{null_free_JD, huang2013eigendecomposition}, a null space-free Jacobi-Davidson eigensolver was introduced specifically for the discretization using Yee's scheme. Another approach, presented in \cite{boffi2006modified}, involves using a mixed finite element method to eliminate the null space.  The kernel compensation method was also developed in \cite{lu2021auxiliary}, building upon edge-conforming elements. {The kernel compensation method shares similarities with the penalized scheme discussed by Monk \cite{petermonk2020}, and is also implicitly reflected in the work of Kikuchi \cite{FUMIO1987509, kikuchi1989mixed}.}

%


In this paper, we propose to apply the kernel compensation method based on the compatible MFD discretization. The MFD method requires additional discretizations for the shifted Nabla operator. The kernel compensation method introduces a compensation operator in the modified auxiliary scheme. This compensation operator fills up the null space of the original eigenproblem with the eigensystem of the compensation operator.  To implement the kernel compensation method, we use a compatible MFD discretization satisfying the discrete de Rham complex chain for the shifted Nabla operator. This discretization ensures that the method is consistent with the underlying mathematical framework and accurately captures the behavior of the shifted Nabla operator.  Compared to the auxiliary scheme based on the curl-conforming finite element \cite{lu2021auxiliary}, the parameter in the compensation operator is mesh size independent. To solve the kernel compensation scheme efficiently, we adopt a simple preconditioner and develop two preconditioner solvers based on the fast Fourier transform (FFT) and the multigrid method, respectively.

			
			The outline of this paper is as follows.  In Section \ref{se:model}, we briefly introduce the eigenproblem of PCs and discuss its quasi-periodic transformations on different lattices. In Section \ref{se:kcm}, we propose the kernel compensation method and establish its mathematical foundations. Section \ref{se:mfd} is devoted to the compatible MFD discretization that satisfies the de Rahm complex, meeting the requirements of the kernel compensation method. A simple preconditioning matrix for the eigensolver is introduced in Section \ref{se:pre}, and the FFT and multigrid solvers are developed for this preconditioning system.  In Section \ref{se:num}, numerical tests for the different lattices are presented to validate the accuracy, effectiveness, and parallel efficiency. Finally, some concluding remarks are given in Section \ref{se:con}.

   \section{Model problem}\label{se:model}
   The mathematical interpretation of three-dimensional photonic crystals (PCs) is given by  time-harmonic Maxwell's equations with the divergence-free conditions:
			\begin{align}\label{time_harmonic_maxwell}
				\left\{
				\begin{aligned}					
					&\nabla\times\vec H=\mathrm{i}\omega\varepsilon\vec E,\ \ \nabla\times\vec E=-\mathrm{i}\omega\mu\vec H,\ \ \ \mbox{in  }\RR^3,\\
                        &\nabla\cdot(\varepsilon\vec E)=0,\ \ \nabla\cdot(\mu\vec H)=0,\ \ \ \mbox{in  }\RR^3,
				\end{aligned}
				\right.
			\end{align}
where $\vec H,\vec E$ are magnetic and electric fields, $\mathrm{i} = \sqrt{-1}$ denotes the imanginary unit   and $\omega$ is the frequency. The parameters $\varepsilon$ and $\mu$ are the permittivity and permeability, respectively. In this article, we assume $\mu\equiv1$ without loss of generality. Thus equation (\ref{time_harmonic_maxwell}) can be rewritten in terms of the magnetic field only,  		
\begin{align}\label{magnetic_maxwell}
				\left\{
					\begin{aligned}						&\nabla\times(\varepsilon^{-1}\nabla\vec\times \vec H)=\omega^2\vec H,\ \ \  \mbox{in  }\RR^3,\\
						&\nabla\cdot\vec H=0, \ \ \ \mbox{in  }\RR^3.
					\end{aligned}
				\right.				
			\end{align}
It can also be written (\ref{time_harmonic_maxwell}) in terms of the electric field as follows
\begin{align}\label{magnetic_maxwell-E}
				\left\{
					\begin{aligned}						&\nabla\times\nabla\vec\times \vec E=\omega^2\varepsilon\vec E,\ \ \ \vec x\in\RR^3,\\
						&\nabla\cdot(\varepsilon\vec E)=0.
					\end{aligned}
				\right.				
			\end{align}
 Even though equations (\ref{magnetic_maxwell}) and (\ref{magnetic_maxwell-E}) are equivalent, it is simpler to develop the kernel compensation method with the model (\ref{magnetic_maxwell}). Hereafter, we will use the model equation (\ref{magnetic_maxwell}). 
                
			The permittivity $\varepsilon$ is such that $\varepsilon=\varepsilon_1$ inside a given material and $\varepsilon=\varepsilon_0$ outside it, with fixed constants $\varepsilon_0,\varepsilon_1>0$. As the material is periodic, we also have
			$$\varepsilon(\vec x+\vec a)=\varepsilon(\vec x), \quad \forall\vec x\in\RR^3,$$ for each $\vec a$ belongs to the Bravias lattice 
			$$ \left\{\sum\limits_{n=1}^3k_n\vec a_n,\ \ k_n\in\ZZ\right\}.$$ Here $\{\vec a_n, n=1,2,3\}$ are lattice translation vectors that span the primitive cell
\begin{equation}\nonumber 
\Omega:=\left\lbrace \vec x \in \RR^3: \vec x = \sum_{n=1}^3 x_n  \vec a_n,
 \; x_n \in [0,1], \;n=1, 2, 3 \right\rbrace.
\end{equation}						
Thus, the magnetic field $\vec H$ satisfies the quasi-periodic condition 
$$\vec H(\vec x+\vec a)=e^{\mathrm{i}\wnk\cdot\vec a}\vec H(\vec x)$$ 
due to Bloch's theorem, where  $\wnk\in\RR^3$ belongs to the first Brillouin zone \cite{PCs_mold}. 

    
To simplify this quasi-periodic boundary condition,  we can define $\vec H_{\wnk}(\vec x):=e^{-\mathrm{i}\wnk\cdot\vec x}\vec H(\vec x)$, then for any $\vec a$ in the Bravias lattice,
			\begin{align}\label{Bloch_thm}
				\vec H_{\wnk}(\vec x+\vec a)=e^{-\mathrm{i}\vec \alpha\cdot(\vec x+\vec a)}\vec H(\vec x+\vec a)=e^{-\mathrm{i}\wnk \cdot\vec x}\vec H(\vec x)=\vec H_{\wnk}(\vec x).
			\end{align}
			By substituting $\vec H(\vec x)=e^{\mathrm{i}\wnk\cdot\vec x}\vec  H_{\wnk}(\vec x)$ into (\ref{magnetic_maxwell}), the periodic model defined on the primitive cell  $\Omega$ is as follows
			\begin{align}\label{curlcurl}
				\left\{
				\begin{aligned}					&\nabla_{\wnk}\times(\varepsilon^{-1}\nabla_{\wnk}\times\vec  H_{\wnk} )=\omega^2\vec  H_{\wnk} ,\ \ \ \vec x\in\Omega,\\
					&\nabla_{\wnk}\cdot\vec  H_{\wnk}=0,\\
					&  \vec  H_{\wnk}(\vec x+\vec a_n)=\vec  H_{\wnk}(\vec x), \quad n=1,2,3.
				\end{aligned}
				\right.
			\end{align}
Here $\nabla_{\wnk}:=\nabla+\mathrm{i}\wnk I \ (I$ is the identity operator) is the shifted nabla operator that naturally defines the shifted curl and divergence operator: 
			\begin{align}\label{shifted_nabla}
				\begin{aligned}					
                \nabla_{\wnk}\times \vec  H_{\wnk}: &=\nabla\times\vec  H_{\wnk}+\mathrm{i}(\wnk\times\vec  H_{\wnk}),\\
				\nabla_{\wnk}\cdot \vec  H_{\wnk}:	&=\nabla\cdot\vec  H_{\wnk}+\mathrm{i}(\wnk\cdot \vec  H_{\wnk}).
				\end{aligned}
			\end{align} 
			
In this paper, we only need to consider the cube domain $\Omega=[0,l]^3$. For a more general case of the lattice, we can use the following coordinate change to transform the problem (\ref{curlcurl}) into a cube. The coordinate change is $\vec x=A\vec y$ and denotes			
			$$\vec H_{A,\wnk}(\vec y):=\vec H_{\wnk}(A\vec y),\ \ A=(\vec a_1,\vec a_2,\vec a_3).$$
Then the periodic condition 
$$\vec  H_{\wnk}(\vec x+\vec a_n)=\vec  H_{\wnk}(\vec x),\quad n=1,2,3$$		becomes 
$$\vec H_{A,\wnk}(\vec y+\vec e_n)=\vec H_{A,\wnk}(\vec y) ,\quad n=1,2,3,$$
where $\vec e_1=(1,0,0)^T,\ \vec e_2=(0,1,0)^T,\ \vec e_3=(0,0,1)^T$.
By introducing $\nabla_{A,\wnk}=\nabla_A+\mathrm{i}\wnk I$ and
    \begin{align}\label{coordinate_change}
        \nabla_{A}=\left(\sum\limits_{j=1}^3b_{j1}\frac{\p}{\p y_j}\ ,\ \sum\limits_{j=1}^3b_{j2}\frac{\p}{\p y_j}\ ,\ \sum\limits_{j=1}^3b_{j3}\frac{\p}{\p y_j}\right)^T.
    \end{align}
with  $b_{ij}=(A^{-1})_{ij}$, equation (\ref{curlcurl}) can be written as
    \begin{align}\label{curlcurl_y}
    \left\{
        \begin{aligned}
            &\nabla_{A,\wnk}\times\left(\varepsilon^{-1}\nabla_{A,\wnk}\times\vec H_{A,\wnk}(\vec y)\right)=\omega^2\vec H_{A,\wnk}(\vec y),\ \ \vec y\in[0,l]^3.\\
        &\nabla_{A,\wnk}\cdot\vec H_{A,\wnk}=0,\\
         & \vec H_{A,\wnk}(\vec y+\vec e_n)=H_{A,\wnk}  (\vec y), \quad n=1,2,3,
        \end{aligned}
        \right.
    \end{align}
    which is in the cube domain again. 
 
 Since the null space of the curl operator contains an infinite-dimensional subspace, the discrete approximation of this operator also has a huge null space whose dimension increases as the grid is refined. From a physical point of view, only the low-energy bands are desirable, i.e., the first few non-zero frequencies of the problem. While eigensolvers compute eigenvalues by ascending order, a huge null space will be a disaster for the numerical method. The motivation of the kernel compensation method is to fill the null space and move them to the back seats.

 \section{Kernel compensation method}\label{se:kcm}
 
  The discrete Maxwell eigenvalue problem in photonic crystals has a huge kernel, which leads to numerical difficulty. In \cite{lu2021auxiliary}, Lu and Xu presented a scheme by adding a penalty term to the edge finite element formulation. It is proved that the penalty term complements the kernel of the discrete Maxwell operator.  The discrete Maxwell eigenproblem can be treated as a Laplace eigenproblem in this approach. Thus, the null space issue can be avoided.  Meanwhile, a mesh size dependent penalty parameter is used, which may increase the condition number of the linear system dramatically.  In this section, we will introduce the kernel compensation method in finite difference discretization, which provides a mesh size independent compensation operator. 

  A finite difference scheme for (\ref{curlcurl_y}) can be written in the matrix form as follows
  \begin{align}\label{first_view_mfdm} 
      \left\{
    \begin{aligned}
        &(\Acal M_0\Acal')\vec H_h=\omega_h^2\vec H_h,\\
        &\Bcal\vec H_h=0.
    \end{aligned}
      \right.
  \end{align}
  where $\vec H_h,\ \omega_h$ are  numerical approximations of $\vec H_{A,\alpha},\ \omega$, respectively. Matrices $\Acal$ and  $\Bcal$ are discrete operaors for $\nabla_{A,\wnk}\times$ and 	$\nabla_{A,\wnk}\cdot$. $\Acal'$ and $\Bcal'$ denote their conjugate transposes. $M_0$ is a diagonal matrix with positive entries arising from the dielectric coefficient $\varepsilon^{-1}$. Supposing the integer $N>0$ refers to the grid size in each direction, we have
  $$
  \Acal,M_0\in\CC^{3N^3\times 3N^3},\ \Bcal\in\CC^{N^3\times 3N^3}.
  $$
   Since the curl operator $\nabla_{A,\wnk}\times$ has an infinite-dimensional null space, the dimension of discrete kernel space $\ker\Acal'$ grows correspondingly, which is $O(N^3)$. 
   
   \subsection{Eigenvalue distributions}
   We will first give the eigenvalue distributions for the matrices $\Acal M_0\Acal'$ and $\Bcal'\Bcal$. 
   
   \begin{lem}\label{lm:kera} Assume $M_0$ is a diagonal matrix with positive entries,    for $\vec u\in \CC^{3N^3}$,      \begin{align*}
   \Acal' \vec u = 0  &\Leftrightarrow  \Acal M_0\Acal' \vec u = 0, \\
   \Bcal \vec u = 0  & \Leftrightarrow  \Bcal'\Bcal \vec u = 0.
  \end{align*} 
   \end{lem}
We use $\ker \Acal$ and $\image \Acal$ to denote the null space and the image of the matrix $\Acal$. This lemma shows that 
    \begin{align*}
    \ker\Acal'=\ker\Acal M_0\Acal', \quad \ker\Bcal=\ker\Bcal'\Bcal,
 \end{align*}	
 and
 \begin{align*}
    &\image\Acal=(\ker\Acal')^{\perp} =(\ker\Acal M_0\Acal')^{\perp} = \image\Acal M_0\Acal', \\
    &\image\Bcal'=(\ker\Bcal)^{\perp} = (\ker\Bcal'\Bcal)^{\perp}= \image\Bcal'\Bcal.
 \end{align*}	
 Denote $N_0^\Acal=\dim\ker\Acal',\ N_0^\Bcal=\dim\ker\Bcal$. We have the eigenvalue distributions as
 \begin{align*}
			&\textbf{eigenvalues of }\Acal M_0\Acal'\ \text{to be}:\ \ \overbrace{0,0,\cdots,0}^{N^\Acal_{0}\text{ zeros}}<\overbrace{\lambda^\Acal_{1}\leq\lambda^\Acal_{2}\leq\cdots\leq\lambda^\Acal_{N^\Acal_{1}}}^{N_{1}^\Acal\text{ nonzeros}},\\
			&\textbf{eigenvalues of }\Bcal'\Bcal\ \text{to be}:\ \ \overbrace{0,0,\cdots,0}^{N^\Bcal_{0}\text{ zeros}}<\overbrace{\lambda^\Bcal_{1}\leq\lambda^\Bcal_{2}\leq\cdots\leq\lambda_{N^\Bcal_{1}}^\Bcal}^{N^\Bcal_{1}\text{ nonzeros}}.
		\end{align*}

 The kernel compensation method is based on a fundamental assumption of the compensation operator:
    \begin{align}\label{assumpt1}
       \Bcal\Acal=0 \; (\image\Acal\subset\ker\Bcal ).
    \end{align}
 Then the decomposition of the space follows:
    \begin{align}\label{ortho_decom}
        \CC^{3N^3}=\image \Acal\oplus\image \Bcal'\oplus \Hcal,\quad \Hcal:=\ker \Acal'\cap\ker \Bcal.
    \end{align}
  The assumption \eqref{assumpt1} and Lemma \ref{lm:kera} implies that 
      \begin{align}
          \CC^{3N^3}=\image (\Acal M_0\Acal')\oplus\image (\Bcal'\Bcal)\oplus \Hcal,\quad \Hcal:=\ker \Acal'\cap\ker \Bcal.
       \end{align}
   We also assume that the dimension of the space $\Hcal$ satisfies 
   \begin{align}\label{assumpt2}
       \quad \text{Either}\ \dim\Hcal=0,\ \text{or}\ \dim\Hcal=O(1),
   \end{align}
   which claims that the dimension of $\Hcal$ will not increase with the mesh size.
 $\Hcal$ coincides with the null space of $\Acal M_0\Acal'+\gamma\Bcal'\Bcal,$ for any penalty number $\gamma>0$.    
   \begin{prop}\label{penalty_aux_correctness}
       Based on the assumptions \eqref{assumpt1} and  \eqref{assumpt2}, we have the following eigenvalue distributions  for $\Acal M_0\Acal'+\gamma \Bcal'\Bcal$
		\begin{align*}
			           \overbrace{0,\cdots,0}^{\dim\Hcal\text{ zeros}}<\lambda_1\leq\lambda_2\leq\cdots\leq\lambda_m\leq\cdots,
		\end{align*}
and it holds that 
$$\{\lambda_1,\cdots,\lambda_m,\cdots\}=\{\lambda_1^\Acal,\cdots,\lambda_{N_1^\Acal}^\Acal\}\cup\{\gamma\lambda_1^\Bcal,\cdots,\gamma\lambda_{N_1^\Bcal}^\Bcal\}.$$
    \end{prop}
    \begin{proof}
        Without loss of generality, we assume that $M_0$ is an identity matrix. By the min-max theorem (a generalized version of the Courant-Fischer min-max theorem \cite{golub2013matrix}): 
        \begin{align*}
            \lambda_m=\min_{\chi\subset \Hcal^{\perp}}\max\limits_{\substack{\vec x\in \chi \\ \|\vec x \| =1}}(\|\Acal'\vec x\|^2+\gamma\|\Bcal\vec x\|^2),
        \end{align*}
       where $\chi$ varies among all $m$-dimenisonal subspaces of $\Hcal^\perp$ and $\|\cdot\|$ is $L^2$ norm of vectors. Each $\chi\subset\Hcal^\perp$ can be decomposed into 
$$
    \chi=\chi_1\oplus\chi_2,\quad \chi_1:=\image\Acal\cap\chi,\quad\chi_2:=\image\Bcal'\cap\chi.
$$
We set $m_i=\dim\chi_i$, thus $m_1+m_2=m.$ Since $\chi_1\subset\ker\Bcal$ and $\chi_2\subset\ker\Acal'$, we have
  \begin{align*}
      \begin{aligned}
\lambda_m&=\min\limits_{\substack{m_1+m_2=m\\ m_1,m_2\in\ZZ_{\geq0}}}\min_{\substack{\dim\chi_1=m_1\\ \dim\chi_2=m_2}}\max\{\|\Acal'\vec x_1\|^2+\gamma\|\Bcal\vec x_2\|^2:\ \vec x_1\in\chi_1,\ \vec x_2\in\chi_2,\ \vec x_1+\vec x_2\in S_{\chi}\}\\
&=\min\limits_{\substack{m_1+m_2=m\\ m_1,m_2\in\ZZ_{\geq0}}}\max\{\lambda_{m_1}^\Acal\|\vec x_1\|^2+\gamma\lambda_{m_2}^\Bcal\|\vec x_2\|^2:\ \|\vec x_1\|^2+\|\vec x_2\|^2=1\}\\
&=\min\limits_{\substack{m_1+m_2=m\\ m_1,m_2\in\ZZ_{\geq0}}}\max\{\lambda_{m_1}^\Acal,\gamma\lambda_{m_2}^\Bcal\},
      \end{aligned}
  \end{align*}
  where we set $\lambda_0^\Acal=\lambda_0^\Bcal=0.$ So by induction on $m$ we get:
\begin{align}\label{lambda_auxi}
      \lambda_m=\text{the $m$-th smallest number among }\{\lambda_1^\Acal,\cdots,\lambda_m^\Acal\}\cup\{\gamma\lambda_1^\Bcal,\cdots,\gamma\lambda_m^\Bcal\}.
\end{align}
This proves the result.
    \end{proof}
    \begin{rmk}\label{pnt_criterion}
    The proof of Proposition \ref{penalty_aux_correctness} also indicates that if $\gamma\lambda_1^\Bcal>\lambda_m^\Acal$, then the eigenvalues of  $\Acal M_0\Acal'+\gamma \Bcal'\Bcal$ satisfy 
    $$\lambda_i=\lambda_i^\Acal,  \quad 1\leq i\leq m.$$
\end{rmk}
     \subsection{The kernel compensation method}
 
 With the discussion on the eigenvalue distributions, we present the kernel compensation auxiliary scheme for solving (\ref{first_view_mfdm}): 
    \begin{align}\label{Ker_comp}
        (\Acal M_0\Acal'+\gamma\Bcal'\Bcal)\vec H_h=\omega_h^2\vec H_h,
    \end{align}
  where $\gamma\Bcal'\Bcal$ is the compensation operator with the penalty number {$\gamma>0$}. Under the  assumptions (\ref{assumpt1}), (\ref{assumpt2}) of the operator $\Acal$ and $\Bcal$, Proposition \ref{penalty_aux_correctness} indicates  the auxiliary scheme \eqref{Ker_comp} is able to reach nonzero eigenvalues of (\ref{first_view_mfdm}) without computing numerous zero eigenvalues. Thus the null space disaster and spurious eigenvalues can be avoided.

  \subsection{Comparison with finite element method}
  Previous discussions are based on the formulation (\ref{first_view_mfdm}) given by finite difference discretization. The kernel compensation method also has a finite element version, as is shown in \cite{lu2021auxiliary}, whose formulation is written by
  \begin{align}\label{first_view_fem}
      \left\{
         \begin{aligned}
             & (\Acal M_0\Acal')\vec x=\omega^2M\vec x,\\
             & \Bcal \vec x=0.
         \end{aligned}
      \right.
  \end{align}
  The existence of a nontrivial mass matrix $M$ brings a slight change to the method. The fundamental assumption (\ref{assumpt1}) becomes
  \begin{align}\label{assumpt1_fem}
      \Bcal M^{-1}\Acal=0.
  \end{align}
  Thus the orthogonal decomposition corresponding with (\ref{ortho_decom}) is
  \begin{align}
      \CC^{3N^3}=\image \Acal\oplus_M\image \Bcal'\oplus_M \Hcal,\quad \Hcal:=\ker \Acal'\cap\ker \Bcal.
  \end{align}
  The notation ``$\oplus_M$" represents that direct sum spaces are mutually $M$-orthogonal, while ``$\oplus$" in (\ref{ortho_decom}) contains orthogonality under Euclidean inner product. The results in \cite{lu2021auxiliary} also demonstrate that, for a sufficiently large $\gamma>0$,
  $$
  (\Acal M_0\Acal'+\gamma\Bcal'\Bcal)\vec x=M\vec x
  $$
  computes the same several smallest eigenvalues as those of (\ref{first_view_fem}). 
  
While the finite element method with kernel compensation exhibits greater universality, it is essential to highlight that the simplicity of the finite difference scheme offers both convenience and advantages during numerical computations. Here, we outline the scenarios in which the finite difference method outperforms finite element discretization:
  \begin{itemize}
  \item Given that the scaling operates at the $O(h^3)$ level, where $h>0$ denotes the grid size, the penalty $\gamma$ in the finite element discretization should align with the $O(h^{-3})$ level. In a numerical experiment detailed in \cite{lu2021auxiliary}, selecting $\gamma=2/h^3$   still yields spurious eigenvalues. In Section \ref{se:mfd}, we will demonstrate that our approach necessitates a significantly reduced penalty for efficacy.

  \item While a finite element method on a uniform grid may require more degrees of freedom, it demonstrates comparable accuracy to the finite difference method in testing. However, the creation of an adaptive mesh in 3D space proves to be excessively expensive.
  
  \item The circularity of the matrix allows for utilizing FFT to perform both preconditioning and matrix multiplication within the finite difference scheme.
\end{itemize}  
  
%
   
   \section{Compatible finite difference discretization}	\label{se:mfd}
   In this section, we present the compatible MFD method for Maxwell's equation with the shifted curl operator and will show that the resulting scheme satisfies the assumptions in Section \ref{se:kcm}.
   
   
		\subsection{MFD discretization of function spaces and operators}
		\label{subse:mfdsp}
\subsubsection{Grid function spaces}		
		We first consider the case $\vec a_n=\vec e_n$. Domain $\Omega=[0,l]^3$ is uniformly divided into $N^3$ small cubes, each in size $h\times h\times h$, with $h=l/N.$ Grid points are $(x_i,y_j,z_k)$ with $x_i=ih,\ y_j=jh$ and $z_k=kh$. The degrees of freedom (DoFs) of a scalar field are located at nodes or at the centers of cells, while the DoFs of a vector field are located either at the centers of edges or at the centers of faces (see Figure \ref{DoFs}). Therefore grid function spaces related to nodes, edges, faces and cells are defined as follows:
        \begin{align}\label{DoFs_grid}
            \begin{aligned}
                &\Ncal_h=\{\phi_{i,j,k}\}\simeq\CC^{N^3},\\
                &\Ecal_h=\{(u_1)_{i-\frac{1}{2},j,k}\}\cup\{(u_2)_{i,j-\frac{1}{2},k}\}\cup\{(u_3)_{i,j,k-\frac{1}{2}}\}\simeq\CC^{3N^3},\\
                &\Fcal_h=\{(v_1)_{i,j-\frac{1}{2},k-\frac{1}{2}}\}\cup\{(v_2)_{i-\frac{1}{2},j,k-\frac{1}{2}}\}\cup\{(v_3)_{i-\frac{1}{2},j-\frac{1}{2},k}\}\simeq\CC^{3N^3},\\
                &\Ccal_h=\{\psi_{i-\frac{1}{2},j-\frac{1}{2},k-\frac{1}{2}}\}\simeq\CC^{N^3},
            \end{aligned}            
        \end{align}
where $1\leq i,j,k\leq N$,  $\phi,\psi$ are scalar fields while $ \vec u=(u_1,u_2,u_3),\ \vec v=(v_1,v_2,v_3)$ are vector fields.  Here the notation $\simeq$ refers to isomorphism, while we identify these sets as complex Euclidean spaces. Due to the periodic boundary condition, we suppose evaluation at $\vec x\in \Omega$ equals that of $\vec x+\vec e_n$ for $n=1,2,3$. We also abbreviate $\phi(x_i,y_j,z_k)$ by $\phi_{i,j,k}$, similar for $\vec u,\vec v$ and $\psi$. 

        \begin{figure}[!ht]
        \begin{center}
            \includegraphics[height=4.5cm]{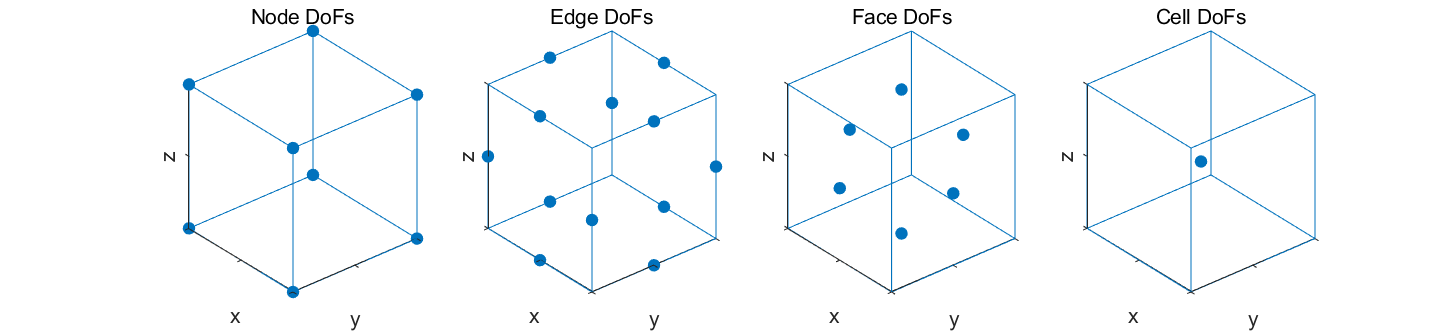}
  \end{center}         
   \caption{DoFs of scalar and vector grid functions on a single cell}
            \label{DoFs}
        \end{figure}

   We point out that $\Ncal_h,\Ecal_h,\Fcal_h$ and $\Ccal_h$ actually approximate certain function spaces related to $L^2$ de Rham complex in $\CC^3$:
 \begin{align}\label{l2_de_rham}
			0\rightarrow H^1(\Omega)\xrightarrow{\nabla_{\wnk}} \vec H(\text{curl})\xrightarrow{\nabla_{\wnk}\times}\vec H(\text{div})\xrightarrow{\nabla_{\wnk}\cdot}L^2(\Omega)\rightarrow0,
		\end{align}
  where the definitions of $H^1(\Omega),\vec H(\text{curl}), \vec H(\text{div})$ and more details can be found in \cite{arnold2000_multigrid}.
  The MFD projection operators that map function spaces in (\ref{l2_de_rham}) to (\ref{DoFs_grid}) are defined as: \begin{align*}
				&\pi_h^{\Ncal}:H^1(\Omega)\rightarrow \Ncal_h,\ \ \ \phi\mapsto\phi_h\text{\ by\ }(\phi_h)_{i,j,k}=\phi_{i,j,k},\\
				&\pi_h^{\Ecal}:H(\text{curl})\rightarrow\Ecal_h,\ \ \ \vec u\mapsto\vec u_h\text{\ by\ }\begin{cases}
				    (\vec u_h)_{i-\frac{1}{2},j,k}=(u_1)_{i-\frac{1}{2},j,k},\\
                    (\vec u_h)_{i,j-\frac{1}{2},k}=(u_2)_{i,j-\frac{1}{2},k}\\
                    (\vec u_h)_{i,j,k-\frac{1}{2}}=(u_3)_{i,j,k-\frac{1}{2}}
				\end{cases},\\
				&\pi_h^{\Fcal}:H(\text{div})\rightarrow\Fcal_h,\ \ \ \vec v\mapsto\vec v_h\text{\ by\ }\begin{cases}
				    (\vec v_h)_{i,j-
        \frac{1}{2},k-\frac{1}{2}}=(v_1)_{i,j-
        \frac{1}{2},k-\frac{1}{2}}\\
        (\vec v_h)_{i-\frac{1}{2},j,k-\frac{1}{2}}=(v_2)_{i-\frac{1}{2},j,k-\frac{1}{2}}\\
        (\vec v_h)_{i-\frac{1}{2},j-\frac{1}{2},k}=(v_3)_{i-\frac{1}{2},j-\frac{1}{2},k}
				\end{cases},\\
				&\pi_h^{\Ccal}:L^2(\Omega)\rightarrow\Ccal_h,\ \ \ \psi\mapsto\psi_h\text{\ by\ }(\psi_h)_{i-\frac{1}{2},j-\frac{1}{2},k-\frac{1}{2}}=\psi_{i-\frac{1}{2},j-\frac{1}{2},k-\frac{1}{2}}.
		\end{align*}

 \subsubsection{Second order discrete operators}

        Next, we construct a discrete version of operators in (\ref{shifted_nabla}): 
		\begin{align}
                \begin{aligned}
                    \nabla_{\wnk}\xrightarrow{\text{discretization}}\ &\grad_{\wnk}=\grad+\mathrm{i}\ \mathbb{I}_{\wnk}^{\Ncal\to\Ecal}:\ \Ncal_h\rightarrow\Ecal_h,\\ (\nabla_{\wnk}\times)\xrightarrow{\text{discretization}}\ &\curl_{\wnk}\ =\curl+\mathrm{i}\ \mathbb{I}_{\wnk}^{\Ecal\to\Fcal}:\ \Ecal_h\rightarrow\Fcal_h,\\ (\nabla_{\wnk}\cdot)\xrightarrow{\text{discretization}}\ &\dive_{\wnk}\ \ =\dive+\mathrm{i}\ \mathbb{I}_{\wnk}^{\Fcal\to\Ccal}:\ \Fcal_h\rightarrow\Ccal_h.
                \end{aligned}
		\end{align}		
		The operators $\grad,\curl,\dive$ approximate $\nabla, \nabla\times, \nabla\cdot$, while  $\mathbb{I}^{\Ncal\to\Ecal}_{\wnk},\mathbb{I}_{\wnk}^{\Ecal\to\Fcal},\mathbb{I}_{\wnk}^{\Fcal\to\Ccal}$  approximate the multipliers $\wnk,(\wnk\times),(\wnk\cdot)$, respectively. Their expressions are given by:
			\begin{align*}
				(\grad\vec\phi_h)_{i-\frac{1}{2},j,k}&=\frac{\phi_{i,j,k}-\phi_{i-1,j,k}}{h},\\
			    (\curl\vec u_h)_{i,j-\frac{1}{2},k-\frac{1}{2}}&=\frac{(u_3)_{i,j,k-\frac{1}{2}}-(u_3)_{i,j-1,k-\frac{1}{2}}}{h}-\frac{(u_2)_{i,j-\frac{1}{2},k}-(u_2)_{i,j-\frac{1}{2},k-1}}{h},\\
				(\dive\vec v_h)_{i-\frac{1}{2},j-\frac{1}{2},k-\frac{1}{2}}&=\frac{(v_1)_{i,j-\frac{1}{2},k-\frac{1}{2}}-(v_1)_{i-1,j-\frac{1}{2},k-\frac{1}{2}}}{h}+\frac{(v_2)_{i-\frac{1}{2},j,k-\frac{1}{2}}-(v_2)_{i-\frac{1}{2},j-1,k-\frac{1}{2}}}{h}\\
				&+\frac{(v_3)_{i-\frac{1}{2},j-\frac{1}{2},k}-(v_3)_{i-\frac{1}{2},j-\frac{1}{2},k-1}}{h},\\
                (\mathbb{I}_{\wnk}^{\Ncal\to\Ecal}\vec\phi_h)_{i-\frac{1}{2},j,k}&=\alpha_1\frac{\phi_{i-1,j,k}+\phi_{i,j,k}}{2},\\
			(\mathbb{I}_{\wnk}^{\Ecal\to\Fcal}\vec u_h)_{i,j-\frac{1}{2},k-\frac{1}{2}}&=\alpha_2\frac{(u_3)_{i,j-1,k-\frac{1}{2}}+(u_3)_{i,j,k-\frac{1}{2}}}{2}-\alpha_3\frac{(u_2)_{i,j-1,k-\frac{1}{2}}+(u_2)_{i,j,k-\frac{1}{2}}}{2},\\
			(\mathbb{I}_{\wnk}^{\Fcal\to\Ccal}\vec v_h)_{i-\frac{1}{2},j-\frac{1}{2},k-\frac{1}{2}}&=\alpha_1\frac{(v_1)_{i-1,j-\frac{1}{2},k-\frac{1}{2}}+(v_1)_{i,j-\frac{1}{2},k-\frac{1}{2}}}{2}+\alpha_2\frac{(v_2)_{i-\frac{1}{2},j-1,k-\frac{1}{2}}+(v_2)_{i-\frac{1}{2},j,k-\frac{1}{2}}}{2}\\
			&+\alpha_3\frac{(v_3)_{i-\frac{1}{2},j-\frac{1}{2},k-1}+(v_3)_{i-\frac{1}{2},j-\frac{1}{2},k}}{2}.
		\end{align*}
 Only part of the expressions are listed here. The rest can be deduced similarly.

 \subsubsection{Higher order discretization}
    Notice that $\grad_{\wnk},\curl_{\wnk},\dive_{\wnk}$ are essentially formed by dimension-wise divided differencing, where the 0th- and 1st-order derivatives in one dimension are approximated by the second-order finite difference approximation. If we substitute it with a high-order difference template, we will obtain a high-order MFD scheme.
    
       The $2k$ order finite difference approximations for the sufficiently smooth function $\phi$ and its derivative $\frac{d\phi}{dx}$ at $x_{j-\frac{1}{2}}$ are the following
       \begin{align}\label{2kfdm}
        (\frac{d\phi}{dx})_{j-\frac{1}{2}}\approx \frac{1}{h}\sum\limits_{s=1}^kc_l(\phi_{j+s-1}-\phi_{j-s}),\ \ \ \ \phi_{j-\frac{1}{2}}\approx\sum\limits_{s=1}^kd_l(\phi_{j+s-1}+\phi_{j-l})
    \end{align}
    where $(c_1,\cdots,c_k),\ (d_1,\cdots,d_k)$ are listed in Table \ref{FD_1D}.
    
    \begin{table}[h]
        \centering   \caption{Symmetric difference stencil in one dimension}
        \label{FD_1D}
  \footnotesize
        \begin{tabular}{|c|l|l|}
            \hline
            $k$ & $(c_1,\cdots,c_k)$ & $(d_1,\cdots,d_k)$ \\
            \hline
             1  & $c_1=1$  &  $d_1=\displaystyle\frac{1}{2}$ \\
            \hline
            2 & $(c_1,c_2)=\left(\displaystyle\frac{9}{8},-\frac{1}{24}\right)$ & $(d_1,d_2)=\left(\displaystyle\frac{9}{16},-\frac{1}{16}\right)$\\
            \hline
            3 & $(c_1,c_2,c_3)=\left(\displaystyle\frac{25}{64},-\frac{25}{384},\frac{3}{640}\right)$ & $(d_1,d_2,d_3)=\left(\displaystyle\frac{75}{128},-\frac{25}{256},\frac{3}{256}\right)$\\
            \hline
            4 & $(c_1,c_2,c_3,c_4)=\displaystyle\left(\frac{1225}{1024},-\frac{245}{3072},\frac{49}{5120},-\frac{5}{7168}\right)$ & $(d_1,d_2,d_3,d_4)=\left(\displaystyle\frac{1225}{2048},-\frac{245}{2048},\frac{49}{2048},-\frac{5}{2048}\right)$\\
            \hline
        \end{tabular}
       \end{table}

    \subsection{Matrix representation}
    \label{subse:m}
    \subsubsection{Matrix representation on simple lattices }
The discrete operators $\grad_{\wnk}$, $\curl_{\wnk}$, $\dive_{\wnk}$ have prominent matrix representations in the form of Kronecker product of circulant matrix and identities. Here $I_N$ denotes the $N\times N$ identity matrix. 
\begin{defn}[Circulant matrix]
    A circulant matrix generated by the vector $\vec a=(a_1,\cdots,a_N)^T\in\RR^{N\times 1}$ is denoted by $\text{Circ}(\vec a)$ with following arrangment of entries: 
    \begin{align}\label{circ_def}
        \text{Circ}(\vec a):=\begin{pmatrix}
            a_1 & a_2 & \cdots & a_{N-1} & a_N\\
            a_N & a_1 & a_2 & \ & a_{N-1}  \\
            \vdots & a_N & a_1 &\ddots &\vdots \\
            a_3 &\ &\ddots &\ddots & a_2 \\
            a_2 & a_3 &\cdots & a_N & a_1
        \end{pmatrix}.
    \end{align}
\end{defn}
 We denote $\mathbb{D}_1=\text{Circ}((-1,1,0,\cdots,0)^T),\ \mathbb{D}_0=\text{Circ}((\frac{1}{2},\frac{1}{2},0,\cdots,0)^T)$ and more explicitly,
		\begin{align}\label{stencils}
                \begin{aligned}
                    \mathbb{D}_1=\begin{pmatrix}
				-1&1&\ & \ &\  \\
				\ &-1&1& \ &\   \\
				\ &\ &\ \ddots &\ddots &\  \\
				\ &\ &\ &-1 &1 \\
				1 &\ &\ & &\ -1
			\end{pmatrix}\in\RR^{N\times N},\ \mathbb{D}_0=\begin{pmatrix}
			\frac{1}{2}&\frac{1}{2}&\ & \ &\  \\
			\ &\frac{1}{2}&\frac{1}{2}& \ &\   \\
			\ &\ &\ \ddots &\ddots &\  \\
			\ &\ &\ &\frac{1}{2} &\frac{1}{2} \\
			\frac{1}{2} &\ &\ & &\ \frac{1}{2}
		\end{pmatrix}\in\RR^{N\times N}.
                \end{aligned}			
		\end{align}
		Using the coefficients  in Table \ref{FD_1D}, notations of circulant matrices $\mathbb{D}_0,\mathbb{D}_1$ can be generalized by
    \begin{align}
        \begin{aligned}
        &\mathbb{D}_1^{(k)}:=\text{Circ}(\vec v_1^{(k)}),\ \ \vec v_1^{(k)}=(-c_1,c_1,c_2,\cdots,c_k,0,\cdots,0,-c_k,\cdots,-c_2)^T\in\RR^{N\times 1},\\
        &\mathbb{D}_0^{(k)}:=\text{Circ}(\vec v_0^{(k)}),\ \ \vec v_0^{(k)}=(d_1,d_1,d_2,\cdots,d_k,0,\cdots,0,d_k,\cdots,d_2)^T\in\RR^{N\times 1}.
        \end{aligned}
    \end{align}
		Then we define the  discrete operator 
		\begin{align}\label{discre}\Dcal_i:=K_i+\mathrm{i} L_i,  \; i=1,2,3
		    \end{align}
		 by the Kronecker product ($\otimes$):
		\begin{align}\label{blocks}
			\begin{aligned}
				&K_1=I_{N}\otimes I_{N}\otimes (\mathbb{D}_1/h),\ K_2=I_{N}\otimes(\mathbb{D}_1/h)\otimes I_{N},\ K_3=(\mathbb{D}_1/h)\otimes I_{N}\otimes I_{N}.\\
                    &L_1=I_{N}\otimes I_{N}\otimes(\alpha_1\mathbb{D}_0),\ \ L_2=I_{N}\otimes(\alpha_2\mathbb{D}_0)\otimes I_{N},\ \ L_3=(\alpha_3\mathbb{D}_0)\otimes I_{N}\otimes I_{N}.
			\end{aligned}			  
		\end{align}
By replacing $\mathbb{D}_0,\mathbb{D}_1$ by $\mathbb{D}_0^{(k)},\mathbb{D}_1^{(k)}$ in (\ref{blocks}), it arrives at a $2k$-th order MFD discretization for (\ref{discre}). In following discussions, we simply write $\mathbb{D}_1:=\mathbb{D}_1^{(k)}$ and $\mathbb{D}_0:=\mathbb{D}_0^{(k)}$.

\begin{lem}\label{abelian_kron}
        If the matrices $A_1,\cdots,A_n$ and $B_1,\cdots,B_n$ are such that $A_iB_i=B_iA_i$ for $1\leq i\leq n$, then $A_1\otimes\cdots\otimes A_n$ and $B_1\otimes\cdots\otimes B_n$ are commutative.   
 \end{lem}
    \begin{proof}
        By  the mixed product property, which  
        can be found  in \cite[Section 4.5.5]{golub2013matrix}, 
        \begin{align*}
            (A_1\otimes\cdots\otimes A_n)(B_1\otimes\cdots\otimes B_n)=(A_1B_1)\otimes\cdots\otimes(A_nB_n).
        \end{align*}
    Since $A_iB_i=B_iA_i$, the conclusion is proved.
    \end{proof}
    
    Since each matrix is commutative with the identity, the following proposition holds.
    \begin{prop}\label{commutativity_of_K}
        The matrices in (\ref{blocks}) satisfy
         $$K_iK_j=K_jK_i,\; L_iK_j=K_jL_i, \; L_iL_j=L_jL_i, \quad i,j\in\{1,2,3\}.$$
    \end{prop}

Matrix representations of $\grad_{\wnk},\curl_{\wnk},\dive_{\wnk}$ follow by
		\begin{align}\label{discrete_2nd_block}
                \grad_{\wnk}=\begin{pmatrix}
				\Dcal_1\\ \Dcal_2\\ \Dcal_3
			\end{pmatrix},\ \ \ \curl_{\wnk}=\begin{pmatrix}
			\ & -\Dcal_3 & \Dcal_2\\ \Dcal_3 & \ & -\Dcal_1\\ -\Dcal_2 & \Dcal_1 &\
		\end{pmatrix},\ \ \ \dive_{\wnk}=\begin{pmatrix}
		    \Dcal_1 & \Dcal_2 & \Dcal_3
		\end{pmatrix}.		
		\end{align}
With a slight abuse of the notations, we will use the same notations to stand for the discrete operators and their representation in matrices in the following. 

 Proposition \ref{commutativity_of_K} implies that
  $$\curl_{\wnk}\ \grad_{\wnk}=0, \quad\dive_{\wnk}\ \curl_{\wnk}=0,$$
 which verifies a discrete version of the de Rham complex (\ref{l2_de_rham}) 
    \begin{align}\label{discrete_de_Rham}
        0\rightarrow \Ncal_h\xrightarrow{\grad_{\wnk}} \Ecal_h\xrightarrow{\curl_{\wnk}}\Fcal_h\xrightarrow{\dive_{\wnk}}\Ccal_h\rightarrow0.
    \end{align}
Thus the condition  \eqref{assumpt1} in Proposition \ref{penalty_aux_correctness} is satisfied. 

\subsubsection{Matrix representation on general lattices}
    So far, we have been discussing the case where $\vec a_n=\vec e_n$. However, it is important to generalize, as not all primitive cells extend their structures along $\{e_n\}$, such as face-centered and body-centered lattices, whose mathematical expressions will be provided in Section \ref{se:num}. Therefore, we will introduce the following coordinate change:
    $$\frac{\p}{\p x_i}=\sum\limits_{j=1}^3b_{ji}\frac{\p}{\p y_j}, \quad \hat K_i=\sum\limits_{j=1}^3b_{ji}K_j \text{\ and\ } \hat\Dcal_i:=\hat K_i+\mathrm{i} L_i.$$ We then replace $\Dcal_i$ in (\ref{discrete_2nd_block}) by $\hat\Dcal_i$ and abbreviate the notation by: 
    \begin{align}\label{unified_mat}
       \grad_{\wnk}=\begin{pmatrix}
				\hat\Dcal_1\\ \hat\Dcal_2\\ \hat\Dcal_3
			\end{pmatrix},\ \ \ \curl_{\wnk}=\begin{pmatrix}
			\ & -\hat\Dcal_3 & \hat\Dcal_2\\ \hat\Dcal_3 & \ & -\hat\Dcal_1\\ -\hat\Dcal_2 & \hat\Dcal_1 &\
		\end{pmatrix},\ \ \ \dive_{\wnk}=\begin{pmatrix}
		    \hat\Dcal_1 & \hat\Dcal_2 & \hat\Dcal_3
		\end{pmatrix}.
    \end{align}
   Since $\hat K_i$ are linear combinations of $K_i$, Propositions \ref{commutativity_of_K}  also holds for $\hat K_i$.
    Therefore we also have 
    $$\curl_{\wnk}\ \grad_{\wnk}=0, \quad\dive_{\wnk}\ \curl_{\wnk}=0$$
    and (\ref{discrete_de_Rham}) is still valid for general lattices. 
    

    \subsection{Derived operators}
    
		Since the rotated magnetic field $\vec v$ induces face-centered DoFs, it requires the definition of a discrete curl operator on $\Fcal_h.$ It is apparent to give an explicit, geometric definition as shown in \cite{null_free_JD}. However, due to the existence of complex components $\mathrm{i}\wnk$, it would be more accurate to define the derived operators corresponding to $\grad_{\wnk},\curl_{\wnk},\dive_{\wnk}$. The derived operators are induced by adjoint operators of $\nabla_{\wnk}$'s as illustrated in \cite{MFDM2016}. We recall the definition of adjoints:
\begin{defn}[Adjoints]
    Given two inner product spaces $(X,\langle\cdot,\cdot\rangle_X), (Y,\langle\cdot,\cdot\rangle_Y)$ and a linear operator $F:X\rightarrow Y,$ the adjoint operator of $F$, denoted by $F^*: Y\rightarrow X$, is such that
        \begin{align}
            \langle Fx,y\rangle_Y=\langle x,F^*y\rangle_X,\ \ \text{for any\ }x\in X,y\in Y.
        \end{align}
\end{defn}
 If both $X,Y$ are Euclidean spaces with finite dimensions and standard scalar products,  $F^*$ is a matrix and $F^*=F'$. In particular, $F^*=F'$ for $F=\grad_{\wnk},\ \curl_{\wnk}, \dive_{\wnk}.$

As is proved in \cite{lu2017discontinuous} there holds:
\begin{align}\label{adjoint_shifted_nabla}
    (\nabla_{\wnk})^*=-\nabla_{\wnk}\cdot,\ \ (\nabla_{\wnk}\times)^*=\nabla_{\wnk}\times,\ \ (\nabla_{\wnk}\cdot)^*=-\nabla_{\wnk}.
\end{align}
  The derived operators denoted by $\widetilde\grad_{\wnk},\ \widetilde\curl_{\wnk},\ \widetilde\dive_{\wnk}$ are defined in a similar format as (\ref{adjoint_shifted_nabla}):
\begin{align}
    \begin{aligned}
        \grad_{\wnk}^*=-\widetilde\dive_{\wnk},\ \ \curl_{\wnk}^*=\widetilde\curl_{\wnk},\ \ \dive_{\wnk}^*=-\widetilde\grad_{\wnk},
    \end{aligned}    
\end{align}
which can be rewritten as 
\begin{align}\label{derived_uniform_grid}
        \begin{aligned}
            &\widetilde\grad_{\wnk}=-\dive_{\wnk}^*=-\dive_{\wnk}':\Ccal_h\rightarrow\Fcal_h,\\
            &\widetilde\curl_{\wnk}=\curl_{\wnk}^*=\curl_{\wnk}':\Fcal_h\rightarrow\Ecal_h,\\ &\widetilde\dive_{\wnk}=-\grad_{\wnk}^*=-\grad_{\wnk}':\Ecal_h\rightarrow\Ncal_h.
        \end{aligned}            
        \end{align}

Furthermore, we have the following result.
   \begin{prop}\label{commutativity_of_K_prime}
        Since $K_iK_j^T=K_j^TK_i,\ L_iL_j^T=L_j^TL_i$ and $K_iL_j^T=L_j^TK_i,$ thus 
        $$\Dcal_i\Dcal_j'=\Dcal_j'\Dcal_i.$$
    \end{prop}
    \begin{proof}
        When $i\neq j$, the conclusion follows by Lemma \ref{abelian_kron}. If $i=j$, it suffices to prove: $\text{Circ}(\vec a)\text{Circ}(\vec b)^T=\text{Circ}(\vec b)^T\text{Circ}(\vec a)$ for any $\vec a,\vec b\in\RR^{N\times 1}.$

        We define the right shift operator $\tau$  by $\tau(a_1,\cdots,a_N)=(a_N,a_1,\cdots,a_{N-1})$ and reverse supscript ``$-1$" by $(a_1,\cdots,a_N)^{-1}=(a_N,a_{N-1},\cdots,a_1)$. We also impose periodicity on the index: $a_{i\pm N}=a_i$. The definition of circulant matrix (\ref{circ_def}) can be rewritten by
        \begin{align*}
            \text{Circ}(\vec a)=(\vec a,\tau\vec a,\cdots,\tau^{N-1}\vec a)^T=(\tau(\vec a^{-1}),\tau^2(\vec a^{-1}),\cdots,\vec a^{-1}).
        \end{align*}
        Then we get
        \begin{align*}
            &[\text{Circ}(\vec a)\text{Circ}(\vec b)^T]_{ij}=(\tau^{i-1} \vec a)^T(\tau^{j-1}\vec b)=(\tau^{i-j}\vec a)^T\vec b,\\
            &[\text{Circ}(\vec b)^T\text{Circ}(\vec a)]_{ij}=(\tau^i(\vec b^{-1}))^T(\tau^j(\vec a^{-1}))=(\tau^{i-j}(\vec b^{-1}))^T\vec a^{-1},
        \end{align*}
        and
        \begin{align*}
            (\tau^{i-j}(\vec b^{-1}))^T\vec a^{-1}&=\sum\limits_{l=1}^N(\vec b^{-1})_{l+i-j}(\vec a^{-1})_l=\sum\limits_{l=1}^Nb_{N+1-(l+i-j)}a_{N+1-l}\\
            &=\sum\limits_{l=1}^Nb_{l-(i-j)}a_l=\sum\limits_{l=1}^Na_{l+i-j}b_l=(\tau^{i-j}\vec a)^T\vec b,
        \end{align*}
     which implies  the $(i,j)$-component of $\text{Circ}(\vec b)^T\text{Circ}(\vec a)$ and $\text{Circ}(\vec a)\text{Circ}(\vec b)^T$ are equal.
    \end{proof}

 Then we have the following corollary of Proposition \ref{commutativity_of_K_prime}. 
    \begin{cor}\label{discrete_vec_lap_mat}
        Denoting
        \begin{align}
            \begin{aligned}
                & \vec L_{\wnk}:=\curl_{\wnk}\widetilde\curl_{\wnk}-\widetilde\grad_{\wnk}\dive_{\wnk}=\curl_{\wnk}\curl_{\wnk}'+\dive_{\wnk}'\dive_{\wnk},\\ 
                & L_{\wnk}:=-\dive_{\wnk}\widetilde\grad_{\wnk}=\dive_{\wnk}\dive_{\wnk}',
            \end{aligned}
        \end{align}
        then we have 
        $$\vec L_{\wnk}=\text{diag}(L_{\wnk},L_{\wnk},L_{\wnk}).$$   
    \end{cor}
This corollary is valid for general lattices, and $\vec L_{\wnk}$ will be adopted as the preconditioner for solving the auxiliary scheme \eqref{Ker_comp}.
    
        
    \subsection{A complete MFD formulation of (\ref{curlcurl_y})}
    To finish the discrete formulation of (\ref{curlcurl_y}), it remains to deal with the multiplier $\varepsilon^{-1}$ arising from the dielectric material. Since $\varepsilon^{-1}$ should be imposed on edge-DoFs, the effect of $\varepsilon^{-1}$ is assumed to be a diagonal matrix, denoted by $M_0$  with positive entries.  We follow the constructions in \cite{lyu2021fame}.
    
As mentioned in Section \ref{se:model}, $\varepsilon=\varepsilon_1$ inside the material and $\varepsilon=\varepsilon_0$ outside. We suppose the $s$-th entry of $M_0$, denoted by $\beta_{ss}$, to be related to the edge-DoF centered at point $\vec r_s,$ then 
    \begin{align*}
        \beta_{ss}=
        \left\{
        \begin{aligned}
            &\varepsilon_1,\ \text{if}\ \vec r_s\text{\ is inside the material,}\\
        &\varepsilon_0,\ \text{otherwise.}
        \end{aligned}        
        \right.
    \end{align*}
If $\vec r_s$ is located right on the interface (singularity of $\varepsilon$), then $\beta_{ss}$ takes the harmonic average of the evaluation of $\varepsilon^{-1}$ at four neighborhood cells. For example, if $\vec r_s=(x_{i-\frac{1}{2}},y_j,z_k),$ then 
    \begin{align}
\beta_{ss}=4/(\varepsilon_{i-\frac{1}{2},j-\frac{1}{2},k-\frac{1}{2}}+\varepsilon_{i-\frac{1}{2},j+\frac{1}{2},k-\frac{1}{2}}+\varepsilon_{i-\frac{1}{2},j-\frac{1}{2},k+\frac{1}{2}}+\varepsilon_{i-\frac{1}{2},j+\frac{1}{2},k+\frac{1}{2}}).             
    \end{align}
 Centers of cells $(x_{i-\frac{1}{2}},y_{j-\frac{1}{2}},z_{k-\frac{1}{2}})$ never insect with the interface in our numerical examples. 
 
 The mimetic finite difference scheme of (\ref{curlcurl_y}) is formulated as follows:
    \begin{align}\label{original_FD_formulation}   
    \left\{
        \begin{aligned}            
        &(\Acal M_0\widetilde\Acal)\vec v_h=\omega_h^2\vec v_h,\\
        &\Bcal\vec v_h=\vec 0,
        \end{aligned}
    \right. 
    \end{align}
    where $\Acal=\curl_{\wnk}, \Bcal=\dive_{\wnk}$ are defined by (\ref{discrete_2nd_block}) or (\ref{unified_mat}) and $\widetilde\Acal=\Acal'$  due to (\ref{derived_uniform_grid}).
    The following proposition indicates the constraint $\Bcal\vec v_h=0$ is redundant when  $\vec v_h\neq \vec 0$.
    
    \begin{prop}
        As long as $\Bcal\Acal=0$ holds and $\vec v_h$ is a nontrivial eigenvector of $\Acal M_0\Acal',$ then discrete divergence-free condition $\Bcal\vec v_h=0$ holds.
    \end{prop}
    \begin{proof}
        For any $\vec x\in \CC^{3N^3},\ \Acal M_0\Acal'\vec x=0$ implies 
       $\vec x'\Acal M_0\Acal'\vec x=0.$  Then $\Acal'\vec x=0$ follows by $M_0$ is positive diagonal.  Therefore we have that $\Acal'\vec x=0$ is equivalent to $\Acal M_0\Acal'\vec x=0$ as in Lemma \ref{lm:kera}.

        Now each nontrivial eigenvector $\vec v_h$ of $\Acal M_0\Acal'$ is such that 
        \begin{align}\label{image_rel}
            \vec v_h\in \image(\Acal M_0\Acal')=\ker(\Acal M_0\Acal')^{\perp}=\ker(\Acal')^{\perp}=\image(\Acal).
        \end{align}
        By the condition $\Bcal\Acal=0$, it follows that $\Bcal\vec v_h=0.$
    \end{proof}

    \subsection{Revisit to the kernel compensation auxiliary scheme (\ref{Ker_comp})}
    The condition \eqref{assumpt1} in Proposition \ref{penalty_aux_correctness} has been verified by the compatible MFD discretization in Section \ref{subse:m}.  The following remark explains the correctness of the assumption  (\ref{assumpt2}). Thus the kernel compensation method works, and we summarize the whole procedure in Algorithm \ref{alg:kcm}.

    \begin{algorithm}[h]
        \caption{Auxiliary scheme with recomputing}\label{alg:kcm}
        \begin{algorithmic}[1]
            \STATE Compute the first $m$ eigenpairs $\{(\lambda_i,\vec x_i)\}_{i=1}^m$ of $\Acal M_0\Acal'+\gamma \Bcal'\Bcal.$
            \STATE Compute $\lambda_{i,\text{re}}=\|M_0^{\frac{1}{2}}\Acal'\vec x_i\|^2/\|\vec x_i\|^2.$
            \IF{$\lambda_{i,\text{re}}=\lambda_{i}$ for $i=1,\cdots,m$}
            \STATE Return $\{(\lambda_i,\vec x_i)\}_{i=1}^m.$
            \ELSE
            \STATE $\%$ \textit{spurious eigenvalues occur}
            \STATE Enlarge $\gamma>0$ (usually set $\gamma:=2\gamma$) and goto line 1.
            \ENDIF
        \end{algorithmic}
    \end{algorithm}
    
   \begin{rmk} 
          When $\vec \alpha\neq 0$ we can see that $\lambda_1^\Bcal\approx\|\wnk\|^2,$ since $\Bcal\Bcal'$ approximates $(\nabla_{\wnk}\cdot)(\nabla_{\wnk}\cdot)^*=-\Delta+\|\wnk\|^2$. Therefore, $\dim\Hcal=0$ and $\gamma=\|\wnk\|^{-2}$.    When $\wnk=0$, according to Corollary \ref{discrete_vec_lap_mat}, we have $\dim\Hcal=3\dim\ker\Bcal'.$ Since $\Bcal'$ approximates $-\nabla$ whose nullity equals to 1, therefore $\dim\Hcal=3$. Meanwhile, $\lambda_1^\Bcal=O(1)$ and thus $\gamma=O(1).$
   \end{rmk}

{Even though $\gamma=O(1)$ or $\|\wnk\|^{-2}$ is theoretically enough, for numerical safety and efficiency, we choose the penalty number to be
	\begin{align}\label{pnt_number}
        \gamma=2\max\{\frac{1}{h},\|\wnk\|^{-2}\}.
    \end{align}
   The numerical results in Section \ref{se:num} show that spurious eigenvalues do not occur with $\gamma$ in (\ref{pnt_number}) for $m\leq20$. The auxiliary scheme for the finite element method \cite{lu2021auxiliary} is using $\gamma=O(h^{-3})$. Our scheme requires a much smaller $\gamma$ and thus better prevents the computation of spurious eigenvalues. Even when spurious eigenvalues emerge, they can be readily identified and discarded through the application of our recomputing Algorithm \ref{alg:kcm}.}

\section{Preconditioners for eigensolvers}\label{se:pre}
In this section, we present two preconditioners for the Locally Optimal Block Preconditioned Conjugate Gradient  (LOBPCG) eigensolver, which is capable of computing multiple eigenpairs.     

    \subsection{Preconditioning}
  The LOBPCG method is genuinely similar to the conjugate gradient method. The latter minimizes $\vec x\mapsto \frac{1}{2}\vec x'A\vec x-\vec b'\vec x$ to solve $A\vec x=\vec b$ while the LOBPCG method optimizes the Rayleigh quotient $\rho(\vec x)=(\vec x'A\vec x)/(\vec x' M\vec x)$ to solve $A\vec x=\lambda M\vec x$ for $(\lambda,\vec x)$, where $A$ and $M$ are give Hermitian matrices. More details can be found in \cite{LOBPCG_Knyazev, knyazev2007block}. 
   The efficiency of the LOBPCG method is highly dependent on the preconditioning. As proved in \cite{neymeyr2006geometric}, the LOBPCG method is technically able to achieve one-step convergence with an optimal preconditioner $\Pcal$. Such a preconditioner is almost impossible to find. But for our auxiliary scheme 
   $$(\Acal M_0\Acal'+\gamma \Bcal'\Bcal)\vec x=\lambda \vec x,$$ 
   a naturally simple and effective preconditioner is  
   $$\Pcal=\Acal \Acal'+\gamma \Bcal'\Bcal.$$
According to \cite{neymeyr2001geometric2}, the convergence speed depends on the condition number of $\Pcal^{-1}(\Acal M_0\Acal'+\gamma\Bcal'\Bcal)$.
 It is equivalent  to considering the extrema of the Rayleigh quotient: 
    \begin{align}
        \rho(\vec x)=\frac{\|M_0^{\frac{1}{2}}\Acal'\vec x\|^2+\gamma\|\Bcal\vec x\|^2}{\|\Acal'\vec x\|^2+\gamma\|\Bcal\vec x\|^2}.
    \end{align}
Since $M_0$ is positive diagonal, we assume $\varepsilon_0,\varepsilon_1$ to be its minimum and maximum entry, respectively. Then we have $\rho(\vec x) \subset[\varepsilon_0,\varepsilon_1]$, so that the condition number of $\Pcal^{-1}(\Acal M_0\Acal'+\gamma\Bcal'\Bcal)$ is less than $\varepsilon_1/\varepsilon_0$, which is much smaller than the condition number of $\Acal M_0\Acal'+\gamma\Bcal'\Bcal.$ 

Using Proposition \ref{commutativity_of_K_prime}, we can write  the preconditioner $\Pcal$ in the matrix form: 
    \begin{align}\label{penalised_marix}
        \Pcal = \begin{pmatrix}
            \gamma\Dcal_1\Dcal1'+\Dcal_2\Dcal_2'+\Dcal_3\Dcal_3' & (\gamma-1)\Dcal_1'\Dcal_2 & (\gamma-1)\Dcal_1'\Dcal_3\\
            (\gamma-1)\Dcal_2'\Dcal_1 & \Dcal_1\Dcal1'+\gamma\Dcal_2\Dcal_2'+\Dcal_3\Dcal_3' & (\gamma-1)\Dcal_2'\Dcal_3\\
            (\gamma-1)\Dcal_3'\Dcal_1 &(\gamma-1)\Dcal_3'\Dcal_2 & \Dcal_1\Dcal1'+\Dcal_2\Dcal_2'+\gamma\Dcal_3\Dcal_3'
        \end{pmatrix}.
    \end{align}

    \begin{rmk}
        In case $\wnk=\vec 0,\ \Acal\Acal'+\Bcal'\Bcal$ is singular. We can impose a shift $c>0$ and solve the shifted problem $(\Acal M_0\Acal'+\gamma \Bcal'\Bcal+c)\vec x=\lambda \vec x$ instead, and subtract $c$ after we obtain the eigenvalues of the shifted problem.
    \end{rmk}

    \subsection{The FFT-based solver}
    According to \cite{chen1987solution_circulant}, every $N\times N$ circulant matrix can be diagonalized by FFT unitary matrix $F\in \CC^{N\times N}$ whose $(i,j)$-th entry is 
    $$F_{ij}=\frac{1}{\sqrt{N}}\omega^{(i-1)(j-1)},\ \omega=\exp(\mathrm{i}\frac{2\pi}{N}).$$
    There holds  $\forall\vec a=(a_1,\cdots,a_N)^T\in\CC^{N\times 1}$
    \begin{align}\label{diag_of_circ}
        \text{Circ}(\vec a)=F\text{diag}(\xi_1,\cdots,\xi_N)F',\ \ \ \xi_i=\sum\limits_{j=1}^Na_j\omega^{(i-1)(j-1)}   .
    \end{align}
We define a Kronecker product of FFT matrices by 
$$\Fcal:=F\otimes F\otimes F.$$
Using Lemma \ref{abelian_kron}, the mixed product theory and (\ref{diag_of_circ}) lead to the following corollary.
     \begin{cor}
       If we assume 
       $$\Lambda_i:=F'\mathbb{D}_iF, \quad  \mbox{for   }  i=0,1$$ and $\Lambda_0,\Lambda_1$ are both diagonal matrices, then  $K_i,M_i$ can simultaneously be diagonalized by $\Fcal$. More precisely, 
        \begin{align}  
        \Fcal'K_1\Fcal=I_{N}\otimes I_{N}\otimes\Lambda_1,\ \Fcal'K_2\Fcal=I_N\otimes \Lambda_1\otimes I_N,\ \Fcal'K_3\Fcal=\Lambda_1\otimes I_{N}\otimes I_{N},\\
        \Fcal'M_1\Fcal=I_{N}\otimes I_{N}\otimes\Lambda_0,\ \Fcal'M_2\Fcal=I_N\otimes \Lambda_0\otimes I_N,\ \Fcal'M_3\Fcal=\Lambda_0\otimes I_{N}\otimes I_{N}.
       \end{align}
Furthermore, $K_iK_j,\ K_iK_j',\ K_iL_j,\ K_iL_j'$ can also be diagonalized by $\Fcal.$ 
    \end{cor}

   Operations involved in solving the matrix in (\ref{penalised_marix}) are computing eigenvalues and applying FFT, inverse FFT. We summarize the procedure in Algorithm 5.1.
    \begin{algorithm}[h!]
        \caption{The FFT-based solver}
        \begin{algorithmic}[1]
           \STATE \textbf{Input:} Diagonal matrix $D\in\mathbb{R}^{n\times n}$, vector $\vec b\in \RR^{n},\ n=N^3.$
           \STATE \textbf{Output:} Solution vector $\vec x$
           \STATE Compute $d_{i,j,k}=d_{1,i}+d_{2,j}+d_{3,k},\ D:=$diag(reshape($(d_{i,j,k}),[\ ],n$)).
           \STATE Reshape: $B=$reshape($\vec b,[N,N,N]$).
           \FOR{$d=1:3$}
                \STATE $B=$fft($B,[\ ],d$); \ \ \ \ \% \textit{apply FFT along each dimension.} 
           \ENDFOR
           \STATE $B=$reshape($D^{-1}$reshape($B,[\ ],n$),$[N,N,N]$);
           \FOR{$d=1:3$}
                \STATE $B=$ifft($B,[\ ],d$); \ \ \ \ \%\textit{apply inverse FFT along each dimension.} 
           \ENDFOR
           \STATE Return reshape($B,[\ ],n$);
        \end{algorithmic}
    \end{algorithm}

    \subsection{The multigrid solver with distributive smoother}

  An iterative multigrid method with the ILU-type smoothing process \cite{TrotMult2001} is shown in Algorithm \ref{MG_ilu}.
    \begin{algorithm}[h!]
        \caption{The multigrid solver  with ILU-type iteration, MG($x^{(0)},\vec b_l,A_l,n_1,n_2$)}
        \label{MG_ilu}
        \begin{algorithmic}[1]
            \STATE\textbf{Input:} Initial guess $\vec x^{(0)},$ level $l$, load vector $\vec b_l$, stiffness matrix $A_l$ and its ILU \\    
            \ \ \ \ \ \ \ \ \ \ \ factorization $A_l\approx L_lU_l.$ Restriction operator $R_l.$
            \STATE \textbf{Output:} Solution $A_l^{-1}\vec b_l.$
            \IF{$l==1$}
                \STATE Return $A_1^{-1}\vec x^{(0)}.$
            \ENDIF
            \STATE\textbf{Pre-smoothing}: Smooth $\vec x^{(i+1)}=\vec x^{(i)}-(L_lU_l)^{-1}(\vec b_l-A_l\vec x^{(i)})$ for $n_1$ times.
            \STATE Compute residual $\vec r=\vec b_l-A_l\vec x^{(n_1)}.$
            \STATE\textbf{Correction}: $\vec x^{(i+1)}=\text{MG}(\vec 0,R_l\vec r,A_{l-1},n_1,n_2).$
            \STATE\textbf{Post-smoothing}: Smooth $\vec x^{(i+1)}=\vec x^{(i)}-(L_lU_l)^{-1}(\vec b_l-A_l\vec x^{(i)})$ for $n_2$ times.
            \STATE Return $\vec x^{(n_1+n_2+1)}.$
        \end{algorithmic}
    \end{algorithm}

  It is important to note that if the system is more ill-conditioned, the ILUs will be less effective and more prone to failure. Therefore it is necessary to apply some pretreatments to the system 
 \begin{align}\label{eq:sys}
 (\Acal \Acal'+\gamma \Bcal'\Bcal)\vec x=\vec f.
   \end{align}
 We introduce the intermediate variable $\vec p=\Bcal\vec u,$ then the system (\ref{eq:sys}) is equivalent to
    \begin{align}
       \begin{pmatrix}
            \Acal \Acal'+(\gamma+1) \Bcal'\Bcal & -\Bcal' \\ -\Bcal & I
        \end{pmatrix}\begin{pmatrix}
            \vec x\\ \vec p
        \end{pmatrix}=\begin{pmatrix}
            \vec f\\ \vec 0
        \end{pmatrix}.
    \end{align}
Then we define matrices 
$$\Lcal=\begin{pmatrix}
            \Acal \Acal'+(\gamma+1) \Bcal'\Bcal & -\Bcal' \\ -\Bcal & I
        \end{pmatrix}
, \quad 
\Lcal_0:=\begin{pmatrix}
        I & \Bcal'\\ \gamma \Bcal & (\gamma+1)L_{\wnk}
    \end{pmatrix},$$ 
    then 
    $$\Lcal\Lcal_0=\begin{pmatrix}
            \vec L_{\wnk} & \ \\ (\gamma-1)\Bcal+c & \gamma L_{\wnk}
        \end{pmatrix}.$$
         To compute the system (\ref{eq:sys}), one may first solve 
         $$\Lcal\Lcal_0\begin{pmatrix}
            \vec y\\ \vec q
        \end{pmatrix} =\begin{pmatrix}
            \vec f\\ \vec 0
        \end{pmatrix}, \mbox{for  } \vec y,\vec q,$$ 
         then compute $$\vec x=\vec y+\Bcal'\vec q.$$
          In this approach,  the ILU factorization will be performed on the matrix with a better condition number. This technique is called the distributive smoother in \cite{gaspar2008distributive}.
 
 Each smoothing step requires the solution of a triangular matrix.  In \cite{SAIT_Lu}, the authors provide a threshold-based sparse approximate inverse for triangular (SAIT) matrices for computing the sparse approximate inverse of triangular matrices. The idea is to use Jacobi's iteration while dropping small entries to ensure sparsity, see \cite{SAIT_Lu} for more details.  Since the ILU-SAIT smoother does not require solving a triangular system, this enables us to implement GPU parallelization in programming. Therefore, we adopt the ILU-SAIT preconditioner as the smoother of the multigrid method for parallel efficiency.


    \section{Numerical experiments}\label{se:num}
    In this section, we will consider the simple cubic (SC), face centered cubic (FCC) and body centered cubic (BCC) lattices of the three-dimensional PCs:
 		\begin{align}\label{translation_vecs}
                    \left\{                    
                    \begin{aligned}
                        &\text{SC lattice:\ }\vec a_1=(l,0,0)^T, \ \ \vec a_2=(0,l,0)^T,\ \ \vec a_3=(0,0,l)^T,\\
                        &\text{FCC lattice:\ }\vec a_1=(0,\frac{l}{2},\frac{l}{2})^T, \ \vec a_2=(\frac{l}{2},0,\frac{l}{2})^T,\ \vec a_3=(\frac{l}{2},\frac{l}{2},0)^T,\\
                        &\text{BCC lattice:\ }\vec a_1=(-\frac{l}{2},\frac{l}{2},\frac{l}{2})^T,\ \vec a_2=(\frac{l}{2},-\frac{l}{2},\frac{l}{2})^T,\ \vec a_3=(\frac{l}{2},\frac{l}{2},-\frac{l}{2})^T.
                    \end{aligned}		
                    \right.
			\end{align}
Here $l>0$ is called the lattice constant. A primitive cell of the SC lattice is the cube $[0,l]^3$. The other two lattices can be transferred to the cube as described in Section \ref{se:model}.
To obtain a complete band gap figure, we do not need to traverse the whole region, instead, based on the symmetry of Brillouin zones \cite{PCs_mold}, only information at symmetry points and on their connecting lines is required. The symmetry points of the Brillouin zone of SC, FCC, and BCC lattices are:
   \begin{align}
        \begin{aligned}
            &\text{SC:\ \ }\Gamma(0,0,0),\ L(\frac{\pi}{l},0,0),\ M(\frac{\pi}{l},\frac{\pi}{l},0),\ N(\frac{\pi}{l},\frac{\pi}{l},\frac{\pi}{l}),\\
            &\text{FCC:\ }X(0,\frac{2\pi}{l},0),\ U(\frac{\pi}{2l},0,\frac{\pi}{2l}),\ L(\frac{\pi}{l},\frac{\pi}{l},\frac{\pi}{l}),\ \Gamma(0,0,0),\ W(\frac{\pi}{l},\frac{2\pi}{l},0),\ K(\frac{3\pi}{2l},\frac{3\pi}{2l},0),\\
        &\text{BCC:\ }H'(0,0,\frac{2\pi}{l}),\ \Gamma(0,0,0),\ P(\frac{\pi}{l},\frac{\pi}{l},\frac{\pi}{l}),\ N(\frac{\pi}{l},0,\frac{\pi}{l}),\ H(0,\frac{2\pi}{l},0).
        \end{aligned}
    \end{align}
We will compute the eigenvalues of SC, FCC, and BCC lattices with both isotropic and anisotropic materials, plot their band gaps, and demonstrate the efficiency of GPU acceleration. All programs are implemented on a platform equipped with an NVIDIA A100-PCIE-40GB GPU. {To facilitate the reproduction and comparison of our numerical results, the source code has been made publicly available on GitHub at \href{https://github.com/Epsilon-79th/linear-eigenvalue-problems-in-photonic-crystals}{https://github.com/Epsilon-79th/linear-eigenvalue-problems-in-photonic-crystals}.}

    \subsection{Isotropic cases: $\varepsilon\equiv1$}
   We use this simple case to demonstrate the accuracy of our method.  We can compute the exact eigenvalues of this problem.
We let $\vec H_{\wnk}(\vec x):=\vec C\exp(\mathrm{i}\frac{2\pi}{l}\vec k\cdot\vec x)$ for some $\vec k\in\ZZ^3$ 
    and $\vec C\in\CC^3$, then
    \begin{align*}
        \left\{
            \begin{aligned}
                &\nabla_{\wnk}\times\nabla_{\wnk}\times\vec H_{\wnk}=\omega^2\vec H_{\wnk},\\
                &\nabla_{\wnk}\cdot\vec H_{\wnk}=0.
            \end{aligned}\right.
    \end{align*}
    becomes
    \begin{align*}
        \left\{
            \begin{aligned}
                &\mathrm{i}(\wnk+\frac{2\pi}{l}\vec k)\times(\mathrm{i}(\wnk+\frac{2\pi}{l}\vec k)\times \vec C)=\omega^2\vec C,\\
                &\mathrm{i}(\wnk+\frac{2\pi}{l}\vec k)\cdot\vec C=0.
            \end{aligned}
        \right.
    \end{align*}
    Therefore the exact eigenvalues $\omega^2$'s are
    \begin{align}\label{exact_eig}
        \omega^2=|\wnk+\frac{2\pi}{l}\vec k|^2,\ \vec k\in\ZZ^3.
    \end{align}

    As the simplest case, we choose lattice constant $l=2\pi$ and Bloch vector $\wnk=(\pi/l,0,0)$.    Errors and convergence rates of the 2nd, 4th, 6th and 8th order schemes are shown in Tables \ref{iso1} - \ref{iso4}. It shows that the desired accuracy can be achieved using the kernel compensation method with compatible MFD schemes.
   
    \begin{table}[!ht]
\centering{
\caption{Accuracy of the 2nd order scheme, isotropic media}
\label{iso1}
\begin{tabular}{|c|c|c|c|c|c|c|c|}
\hline
$N$& $N=10$  & \multicolumn{2}{c|}{$N=20$}  & \multicolumn{2}{c|}{$N=40$}  & \multicolumn{2}{c|}{$N=80$} \\ 
 \hline
 $\omega^2$ & $|\omega^2-\omega_h^2|$ & $|\omega^2-\omega_h^2|$ & ord & $|\omega^2-\omega_h^2|$ & ord & $|\omega^2-\omega_h^2|$ & ord\\ 
 \hline
0.25 & 4.11e-15  & 7.22e-16 & \  & 1.11e-16 & \  & 1.11e-16 & \ \\ 
\hline
0.25 & 1.00e-14  & 2.70e-14 & \  & 2.25e-14 & \  & 1.11e-16 & \ \\ 
\hline
0.25 & 8.17e-03  & 2.05e-03 &     1.99 & 5.14e-04 &     2.00 & 1.28e-04 &     2.00\\ 
\hline
0.25 & 8.17e-03  & 2.05e-03 &     1.99 & 5.14e-04 &     2.00 & 1.28e-04 &     2.00\\ 
\hline
1.25 & 3.25e-02  & 8.20e-03 &     1.99 & 2.05e-03 &     2.00 & 5.14e-04 &     2.00\\ 
\hline
1.25 & 3.25e-02  & 8.20e-03 &     1.99 & 2.05e-03 &     2.00 & 5.14e-04 &     2.00\\ 
\hline
\end{tabular}
}
\end{table}

\begin{table}[!ht]
\centering{
\caption{Accuracy table of the 4th order scheme, isotropic media}
\label{iso2}
\begin{tabular}{|c|c|c|c|c|c|c|c|}
\hline
$N$& $N=10$  & \multicolumn{2}{c|}{$N=20$}  & \multicolumn{2}{c|}{$N=40$}  & \multicolumn{2}{c|}{$N=80$} \\ 
 \hline
 $\omega^2$ & $|\omega^2-\omega_h^2|$ & $|\omega^2-\omega_h^2|$ & ord & $|\omega^2-\omega_h^2|$ & ord & $|\omega^2-\omega_h^2|$ & ord\\ 
 \hline
0.25 & 3.77e-15  & 8.33e-16 & \  & 5.55e-17 & \  & 2.22e-16 & \ \\ 
\hline
0.25 & 1.45e-14  & 2.56e-14 & \  & 1.67e-16 & \  & 5.55e-17 & \ \\ 
\hline
0.25 & 1.05e-03  & 6.78e-05 &     3.96 & 4.27e-06 &     3.99 & 2.67e-07 &     4.00\\ 
\hline
0.25 & 1.05e-03  & 6.78e-05 &     3.96 & 4.27e-06 &     3.99 & 2.67e-07 &     4.00\\ 
\hline
1.25 & 1.43e-03  & 9.08e-05 &     3.97 & 5.70e-06 &     3.99 & 3.57e-07 &     4.00\\ 
\hline
1.25 & 1.43e-03  & 9.08e-05 &     3.97 & 5.70e-06 &     3.99 & 3.57e-07 &     4.00\\ 
\hline
\end{tabular}
}
\end{table}

   \begin{table}[!ht]
\centering{
\caption{Accuracy table of the 6th order scheme, isotropic media}
\label{iso3}
\begin{tabular}{|c|c|c|c|c|c|c|c|}
\hline
$N$& $N=10$  & \multicolumn{2}{c|}{$N=20$}  & \multicolumn{2}{c|}{$N=40$}  & \multicolumn{2}{c|}{$N=80$} \\ 
 \hline
 $\omega^2$ & $|\omega^2-\omega_h^2|$ & $|\omega^2-\omega_h^2|$ & ord & $|\omega^2-\omega_h^2|$ & ord & $|\omega^2-\omega_h^2|$ & ord\\ 
 \hline
0.25 & 4.00e-15  & 1.11e-15 & \  & 2.22e-16 & \  & 4.44e-16 & \ \\ 
\hline
0.25 & 8.44e-15  & 1.57e-14 & \  & 1.11e-16 & \  & 1.67e-16 & \ \\ 
\hline
0.25 & 1.00e-04  & 1.65e-06 &     5.93 & 2.61e-08 &     5.98 & 4.09e-10 &     6.00\\ 
\hline
0.25 & 1.00e-04  & 1.65e-06 &     5.93 & 2.61e-08 &     5.98 & 4.09e-10 &     6.00\\ 
\hline
1.25 & 8.18e-05  & 1.33e-06 &     5.95 & 2.09e-08 &     5.99 & 3.27e-10 &     6.00\\ 
\hline
1.25 & 8.18e-05  & 1.33e-06 &     5.95 & 2.09e-08 &     5.99 & 3.27e-10 &     6.00\\ 
\hline
\end{tabular}
}
\end{table}

    \begin{table}[!ht]
\centering{
\caption{Accuracy table of the 8th order scheme, isotropic media}
\label{iso4}
\begin{tabular}{|c|c|c|c|c|c|c|c|}
\hline
$N$& $N=10$  & \multicolumn{2}{c|}{$N=20$}  & \multicolumn{2}{c|}{$N=40$}  & \multicolumn{2}{c|}{$N=80$} \\ 
 \hline
 $\omega^2$ & $|\omega^2-\omega_h^2|$ & $|\omega^2-\omega_h^2|$ & ord & $|\omega^2-\omega_h^2|$ & ord & $|\omega^2-\omega_h^2|$ & ord\\ 
 \hline
0.25 & 5.11e-15  & 6.66e-16 & \  & 1.11e-16 & \  & 2.22e-16 & \ \\ 
\hline
0.25 & 1.14e-14  & 1.09e-14 & \  & 2.78e-17 & \  & 5.55e-16 & \ \\ 
\hline
0.25 & 9.16e-06  & 3.85e-08 &     7.89 & 1.53e-10 &     7.97 & 6.00e-13 &     7.99\\ 
\hline
0.25 & 9.16e-06  & 3.85e-08 &     7.89 & 1.53e-10 &     7.97 & 6.01e-13 &     7.99\\ 
\hline
1.25 & 5.35e-06  & 2.21e-08 &     7.92 & 8.75e-11 &     7.98 & 3.12e-13 &     8.13\\ 
\hline
1.25 & 5.35e-06  & 2.21e-08 &     7.92 & 8.75e-11 &     7.98 & 3.04e-13 &     8.17\\ 
\hline
\end{tabular}
}
\end{table}

    \begin{figure}[!ht]
        \centering        
        \includegraphics[height=6cm]{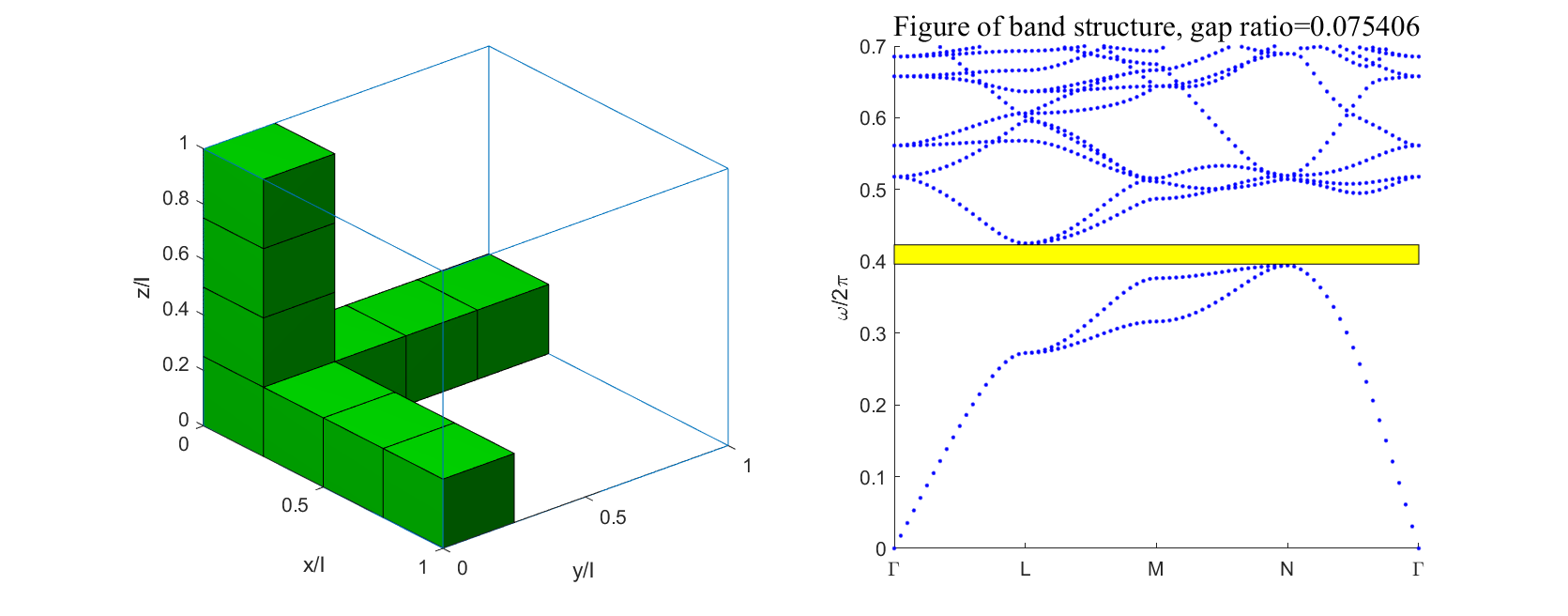}
        \caption{Geometric structure of a single lattice and its band gap in Section \ref{sec:cubic}. $\varepsilon_1/\varepsilon_0=13.$ Grid size $N=100.$ The ratio of the band gap is 0.075406.}
        \label{model1}
    \end{figure}

    \begin{figure}[!ht]
        \centering        
        \includegraphics[height=6cm]{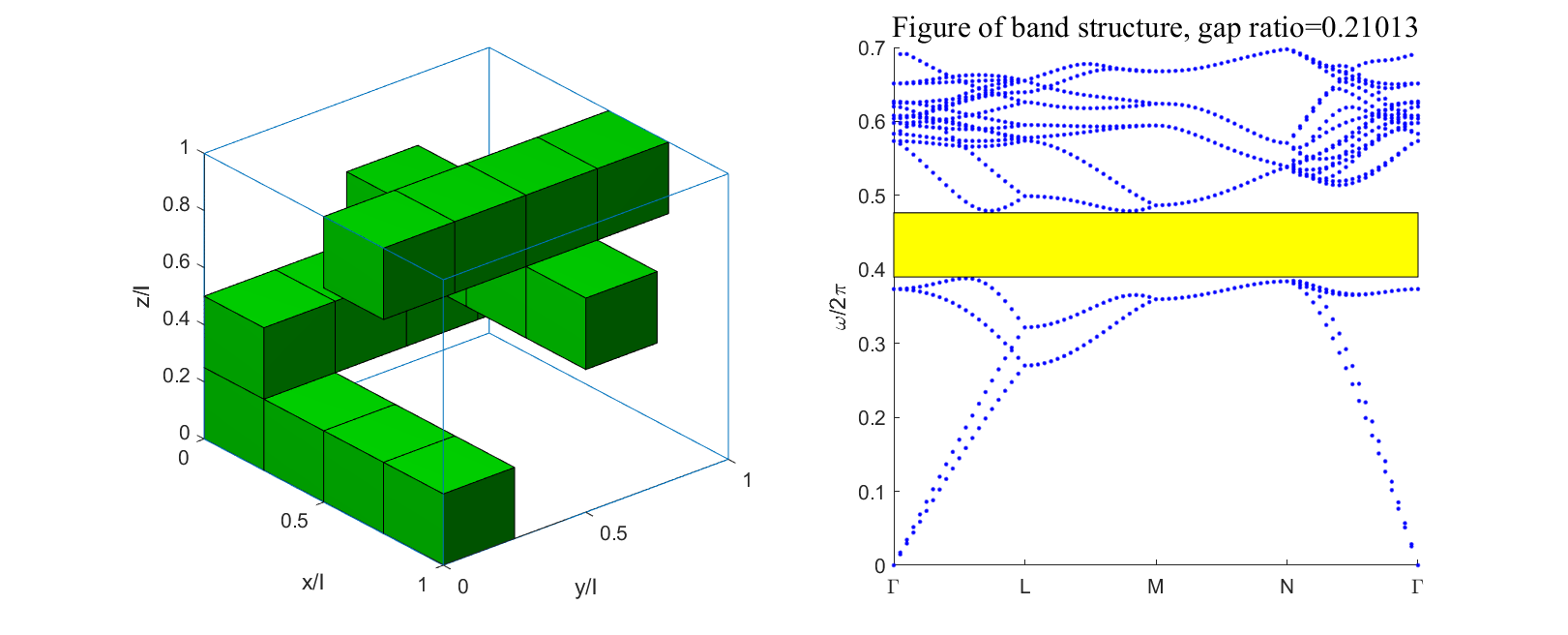}
        \caption{Geometric structure of a single lattice and its band gap in Section \ref{sec:cubic}. $\varepsilon_1/\varepsilon_0=13.$ Grid size $N=100.$ The ratio of the band gap is 0.21013.}
        \label{model2}
    \end{figure}

    \subsection{Cubic lattices}\label{sec:cubic}
    First, we consider the material shown in the left part of Figure \ref{model1} - \ref{model2}. The dielectric function $\varepsilon$ is a piecewise constant that equals to $\varepsilon_{1}$ inside the material and $\varepsilon_0$ outside, with $\varepsilon_1/\varepsilon_0=13$.  The entry of $M_0$ is set to $\varepsilon_1^{-1}$ if the corresponding edge DoF is inside the material, otherwise, it is set to $\varepsilon_0^{-1}$. The ratio of the band gap is defined by
    \begin{align}
        \text{ratio}:=\frac{\omega_{\text{up}}-\omega_{\text{min}}}{(\omega_{\text{up}}+\omega_{\text{min}})/2},
    \end{align}
where $\omega_{\text{up}}$ is the minimum point above the gap while $\omega_{\text{low}}$ the maximum point below the gap. With a larger band gap, the material is able to block light over a wider frequency range. Here are two examples of non-homogeneous materials with flat interfaces. The geometric structure of a single lattice is shown in Figures \ref{model1} - \ref{model2} along with the band gap. The results are obtained when the LOBPCG iteration stops at a relative error of less than $10^{-5}$ using the 2nd-order MFD discretization. These numerical results verify the results obtained in \cite{lu2017discontinuous, lu2021auxiliary}.

    Next, we consider the anisotropic SC lattice with a curved interface. The material domain in which the expression is as follows
    \begin{align}
        \begin{aligned}
            &\big\{(x-\frac{l}{2})^2+(y-\frac{l}{2})^2+(z-\frac{l}{2})^2\leq (0.345l)^2\big\}\cup\big\{(y-\frac{l}{2})^2+(z-\frac{l}{2})^2\leq (0.11l)^2\big\}\\
        &\big\{(x-\frac{l}{2})^2+(z-\frac{l}{2})^2\leq (0.11l)^2\big\}\cup \big\{(x-\frac{l}{2})^2+(y-\frac{l}{2})^2\leq (0.11l)^2\big\}\subset[0,l]^3.
        \end{aligned}        
    \end{align}
The dielectric function $\varepsilon$ remains the same as in the last two cases. The band gap is shown in Figure \ref{model_curv1}, which corresponds to the result in  \cite{null_free_JD, lyu2021fame}.

   \begin{figure}[!ht]
        \centering        
        \includegraphics[height=6cm]{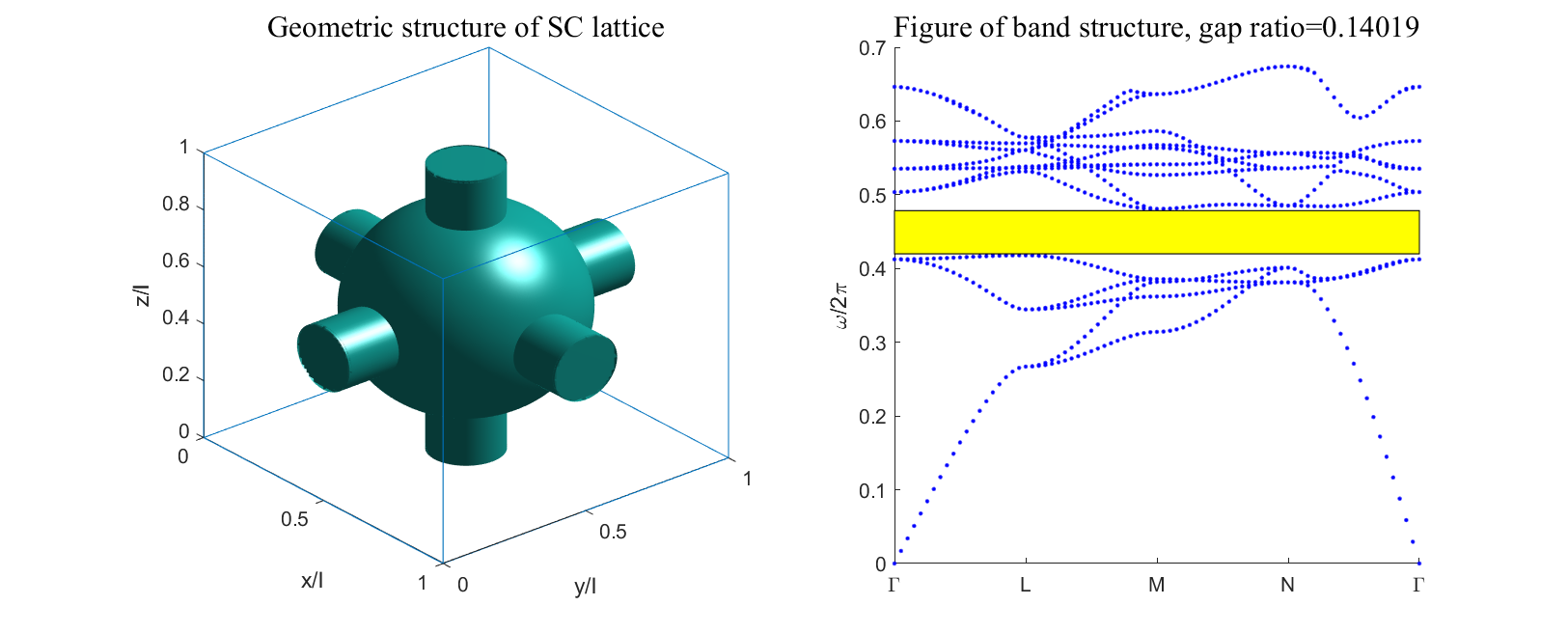}
        \caption{Geometric structure of a single lattice with a curved interface and its band gap in Section \ref{sec:cubic}. Radii of the center sphere and cylinders are $0.345l$ and $0.11l$. $\varepsilon_1/\varepsilon_0=13.$ Grid size $N=150.$ The ratio of the band gap is 0.14019.}
        \label{model_curv1}
    \end{figure}

   \subsection{BCC and FCC lattices} \label{sec:bcc_fcc}
   \begin{figure}[!ht]
        \centering        
        \includegraphics[height=6cm]{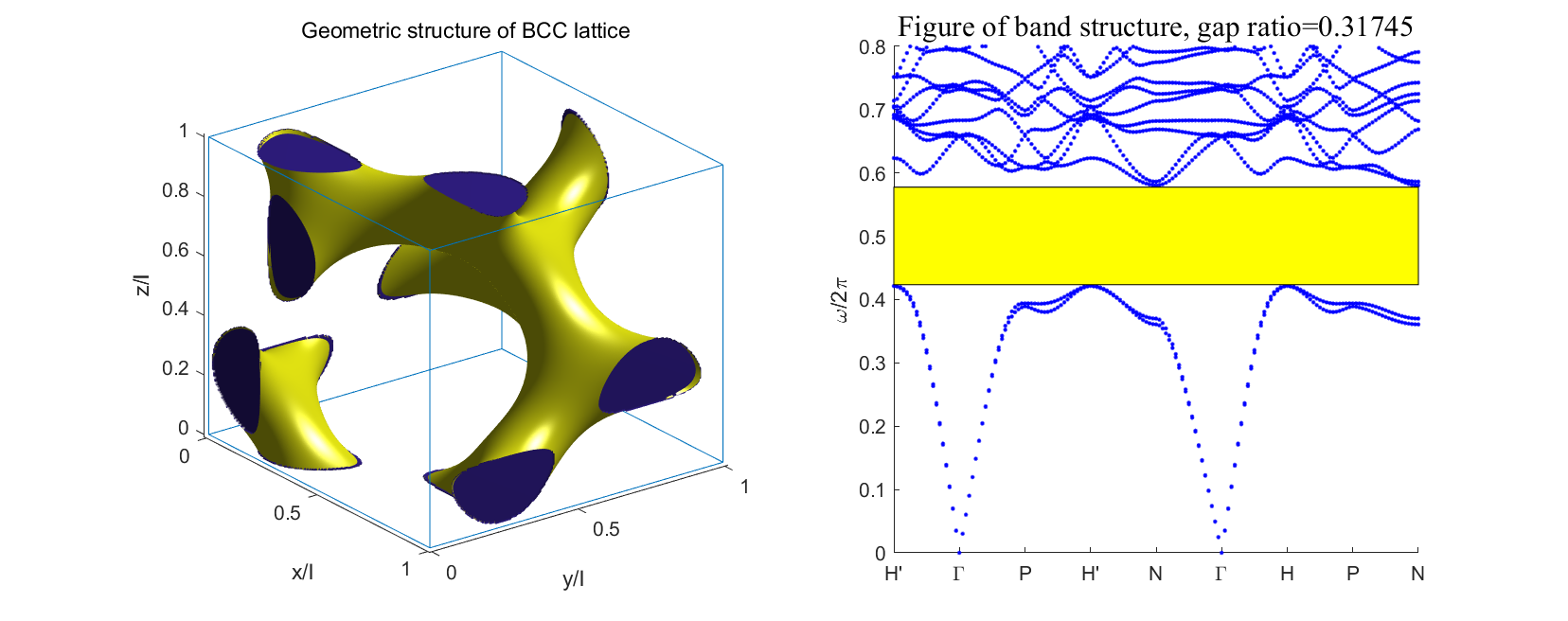}
        \caption{Geometric structure of the BCC lattice with a curved interface and its band gap in Sec. \ref{sec:bcc_fcc}, $\varepsilon_1/\varepsilon_0=16.$ Grid size $N=150.$ The ratio of the band gap is 0.31745.}
        \label{model_bcc}
    \end{figure}

    \begin{figure}[!ht]
        \centering        
        \includegraphics[height=6cm]{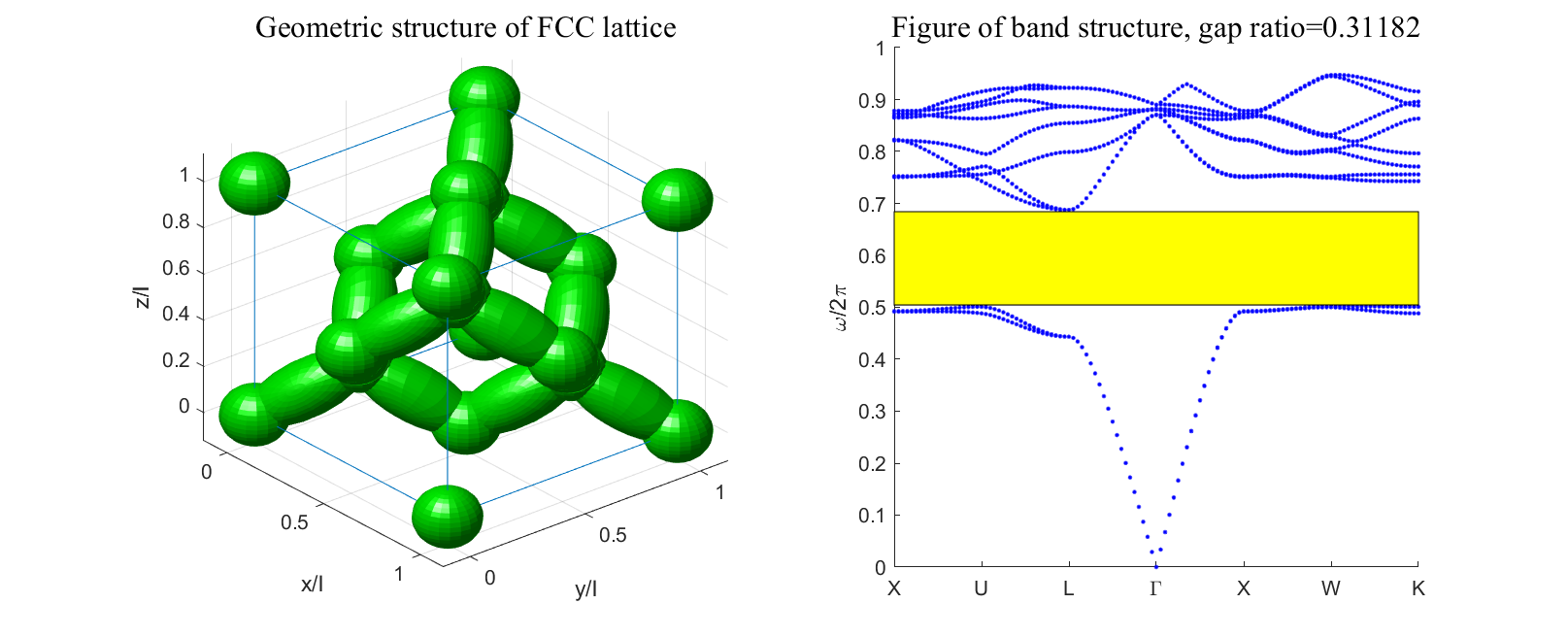}
        \caption{Geometric structure of the FCC lattice with a curved interface and its band gap in Sec. \ref{sec:bcc_fcc}, $\varepsilon_1/\varepsilon_0=13.$ Grid size $N=150.$ The ratio of the band gap is 0.31182.}
        \label{model_fcc}
    \end{figure}

  Now, we consider the examples of the BCC lattice in \cite{lu2013weyl} and the FCC lattice in \cite{huang2013eigendecomposition}. The diagonal matrix $M_0$ arising from $\varepsilon^{-1}$ is chosen to be the same as in Section \ref{sec:cubic}. Translation vectors of BCC and FCC lattices are given in (\ref{translation_vecs}).    

    We recall that $l>0$ refers to the lattice constant. The structure of BCC lattice in our example is a single gyroid \cite{lu2013weyl} that can be approximated by the set $\{g(x,y,z)>1.1\}$ with 
    \begin{align*}
        g(x,y,z)=\sin(\frac{2\pi}{l}x)\cos(\frac{2\pi}{l}y)+\sin(\frac{2\pi}{l}y)\cos(\frac{2\pi}{l}z)+\sin(\frac{2\pi}{l}z)\cos(\frac{2\pi}{l}x),
    \end{align*}
which is shown in Figure \ref{model_bcc}.  

The FCC lattice in Figure \ref{model_fcc} is a diamond structure with $sp^3$-like configuration consisting of dielectric spheres and connecting spheroids. The radius of spheres is $r=0.12l$. The spheroids have their foci at the centers of spheres and minor axis of length $b=0.11l$.  A detailed description of the geometric structure of the FCC lattice can be found in \cite{huang2013eigendecomposition}. Again, the band gaps of the BCC and FCC lattices shown on the right in Figures \ref{model_bcc} and \ref{model_fcc} agree with the results in \cite{huang2013eigendecomposition, lu2013weyl}.

    \subsection{Efficiency comparion with preconditioner}

 Next, in Tables \ref{time_comp_sc_iso} - \ref{time_comp_fc}, we present CPU and GPU  time for different lattices and preconditioner solvers.
We assume $\wnk=(\pi,\pi,\pi)/l$ for all lattices above.

        \begin{table}[!ht]
  \caption{CPU/GPU time of first 10 eigenvalues and acceleration ratio of SC lattices: isotropic material ($\varepsilon\equiv1$) in Table \ref{iso4}. }
              \label{time_comp_sc_iso}
       \centering
               \begin{tabular}{|c|c|c|c|c|}
        \hline
            \multirow{2}{*}{DoFs}  & \multicolumn{4}{c|}{Preconditioner solver: FFT. }  \\
           \cline{2-5}
             & Steps & GPU time & CPU time & Speed up \\
            \hline 
            $3\times 100^3$ & 13 & 11.41s & 238.82s & 20.93  \\
            \hline 
            $3\times 120^3$ & 13 & 19.69s & 374.61s & 19.03   \\
            \hline
            $3\times 150^3$ & 12 & 35.98s & 726.08s & 20.18  \\
            \hline
            \multirow{2}{*}{DoFs}  & \multicolumn{4}{c|}{Preconditioner solver: multigrid.} \\
           \cline{2-5}
             & Steps & GPU time & CPU time & Speed up  \\
            \hline
             $3\times 96^3$ & 13 & 22.19s &1712.09s & 77.16  \\
            \hline 
            $3\times 112^3$ & 12 & 31.66s &2775.40s & 87.66   \\
            \hline
            $3\times 128^3$ & 12 & 45.98s & 4128.31s & 89.79  \\
           \hline
            \end{tabular}
        
%
     \end{table}
    
           \begin{table}[!ht]
        \caption{CPU/GPU time of first 10 eigenvalues and acceleration ratio of SC lattices: anisotropic material in Figure \ref{model_curv1}, $\wnk=(\pi,\pi,\pi)/l.$ }
        \centering
        \begin{tabular}{|c|c|c|c|c|}
        \hline
            \multirow{2}{*}{DoFs}  & \multicolumn{4}{c|}{Preconditioner solver: FFT.}  \\
           \cline{2-5}
             & Steps & GPU time & CPU time & Speed up \\
            \hline 
            $3\times 100^3$ &  44 & 21.68s & 405.96s & 18.73 \\
            \hline
            $3\times 120^3$ & 43 & 38.12s & 746.12s & 19.57 \\
            \hline
            $3\times 150^3$ &  44 & 71.10s & 1217.73s & 17.13 \\
            \hline
            \multirow{2}{*}{DoFs}  & \multicolumn{4}{c|}{Preconditioner solver: multigrid.} \\
           \cline{2-5}
             & Steps & GPU time & CPU time & Speed up  \\
            \hline
            $3\times 96^3$ &  43 & 71.11s & 5598.49s & 78.73 \\
            \hline
            $3\times 112^3$ & 45 & 109.00s & 8186.02s & 75.05 \\
            \hline
            $3\times 128^3$ &  44 & 162.10s & 12917.07s & 79.69 \\
            \hline
            \end{tabular}
 
         \label{time_comp_sc_anti}
    \end{table}
   
    \begin{table}[!ht]
        \centering
  \caption{CPU/GPU time of first 10 eigenvalues and acceleration ratio of BCC lattices in Figures \ref{model_bcc},\ $\wnk=(\pi,\pi,\pi)/l$.}
           \label{time_comp_bc}
           \begin{tabular}{|c|c|c|c|c|}
        \hline
            \multirow{2}{*}{DoFs}  & \multicolumn{4}{c|}{Preconditioner solver: FFT.}  \\
           \cline{2-5}
             & Steps & GPU time & CPU time & Speed up \\
            \hline 
            $3\times 100^3$ &  70 & 40.55s & 915.51s & 22.58 \\
            \hline
            $3\times 120^3$ & 65 & 54.92s & 1100.75s & 20.04 \\
            \hline
            $3\times 150^3$ &  70 & 132.31s & 3052.50s & 23.07 \\
            \hline
            \multirow{2}{*}{DoFs}  & \multicolumn{4}{c|}{Preconditioner solver: multigrid.} \\
           \cline{2-5}
             & Steps & GPU time & CPU time & Speed up  \\
            \hline
            $3\times 96^3$ &  71 &115.06s & 9024.70s & 78.43 \\
            \hline
            $3\times 112^3$ & 71 & 161.08s & 13814.70s & 85.76 \\
            \hline
            $3\times 128^3$ &  70 & 239.55s & 22041.10s & 92.01 \\
            \hline
            \end{tabular}
    
    \end{table}

       \begin{table}[!ht]
        \centering
    \caption{CPU/GPU time of first 10 eigenvalues and acceleration ratio of FCC lattices in Figures \ref{model_fcc},\ $\wnk=(\pi,\pi,\pi)/l$.}
          \label{time_comp_fc}
         \begin{tabular}{|c|c|c|c|c|}
        \hline
            \multirow{2}{*}{DoFs}  & \multicolumn{4}{c|}{Preconditioner solver: FFT.}  \\
           \cline{2-5}
             & Steps & GPU time & CPU time & Speed up \\
            \hline 
            $3\times 100^3$ &  56 & 34.02s & 716.15s & 21.05 \\
            \hline
            $3\times 120^3$ & 54 & 46.06s & 945.12s & 20.52 \\
            \hline
            $3\times 150^3$ &  55 & 105.98s & 2186.37s & 20.63 \\
            \hline
            \multirow{2}{*}{DoFs}  & \multicolumn{4}{c|}{Preconditioner solver: multigrid.} \\
           \cline{2-5}
             & Steps & GPU time & CPU time & Speed up  \\
            \hline
            $3\times 96^3$ &  58 & 92.71s & 7640.12s & 82.41 \\
            \hline
            $3\times 112^3$ & 59 & 128.33s & 11482.41s & 89.48 \\
            \hline
            $3\times 128^3$ &  55 & 189.22s & 17869.88s & 94.44 \\
            \hline
            \end{tabular}
    
     \end{table}

These comparisons show that the computational time of the method can be reduced by a factor of about 20 and 80 for the FFT and multigrid preconditioner solvers, respectively. The FFT solver has excellent CPU efficiency, and the effects of GPU acceleration are relatively insignificant. The multigrid solver, on the other hand, can be dramatically accelerated by the GPU, but it takes an even longer time to complete a single iteration. Meanwhile, regardless of which preconditioner solver is used, the iteration times do not differ and appear to be independent of the grid size, which is only related to the problem itself.
Our examples of FCC and BCC are derived from the numerical experiments in \cite{lyu2021fame} and \cite{null_free_JD}. Through preconditioning using FFT, our algorithm demonstrates reduced runtime, even without GPU parallelization.

    \section{Conclusions}\label{se:con}
    In this paper, we have presented a kernel compensation method based on the compatible mimetic finite difference method for the eigenproblems of photon crystals without passing through numerous zero eigenvalues. When solving the energy band gap of PCs, these zero eigenvalues increase with the mesh size and thus cause many problems for the eigensolvers. The introduction of the kernel compensation operator fills up the kernel of the eigenproblems, which is based on the compatible discrete chain complex in the MFD discretization. The compensation operators have been developed for the shifted curl operator in the SC, BCC, and FCC lattices introduced by the Floquet-Bloch theory. {This method is specifically designed for photonic crystals with periodic structures.} Two parallel effective preconditioner solvers based on the FFT method and the multigrid method have been proposed and tested. 


%


\bibliographystyle{abbrv}
\bibliography{ref}{}

\end{document}